\documentclass{amsart}

\usepackage{lipsum}

\newcommand\blfootnote[1]{%
  \begingroup
  \renewcommand\thefootnote{}\footnote{#1}%
  \addtocounter{footnote}{-1}%
  \endgroup
}

\usepackage{geometry}
\usepackage{amsmath}
\usepackage{amsthm}
\usepackage{amsfonts}

\newtheorem{theorem}{Theorem}[section]

\newtheorem{lemma}{Lemma}[section]
\newtheorem{proposition}{Proposition}[section]
\newtheorem{corollary}{Corollary}[section]
\newtheorem{example}{Example}[section]
\newtheorem{definition}{Definition}[section]
\newtheorem{remark}{Remark}[section]

\newcommand\N{\mathbb{N}}
\newcommand{\R}{\mathbb R}

 \newcommand{\F}{\mathcal F}
 \newcommand{\B}{\mathcal B}
 \newcommand{\T}{\mathcal{TB}}    
    
 \newcommand{\C}{\mathcal{C}}  
  \newcommand{\BC}{\mathcal{BC}}  
  \newcommand{\TBC}{\mathcal{TBC}}

  \newcommand{\X}{\mathcal X}
  \newcommand{\V}{\mathcal V}

   \newcommand{\K}{\mathcal K}
  \newcommand{\D}{\mathcal D}

  \newcommand{\Linear}{\mathcal L}
  
  \newcommand{\U}{K} 
  
  \newcommand{\Tm}{T} 

 \newcommand{\W}{W}  

 \newcommand{\gammaBCk}{\overline{\gamma}}
\newcommand{\alphaBCk}{\overline{\alpha}}

  \newcommand{\M}{M}
  \newcommand{\w}{\omega_\B}

 \newcommand{\eps}{\varepsilon}
  
 \newcommand{\z}{w}

\begin{document}
\title[]{Regular measures of noncompactness  and Ascoli-Arzel\` a type compactness criteria in spaces of vector-valued  functions}

\author{D. Caponetti}


\author{A. Trombetta}
\author{G. Trombetta}

\maketitle

 \baselineskip=15pt






\begin{abstract} 
 In this paper we estimate the Kuratowski and the Hausdorff measures of noncompactness of bounded subsets of spaces of vector-valued bounded functions  and  
of  vector-valued bounded differentiable functions. To this end, we use a quantitative characteristic modeled on a new  equiconti\-nuity-type concept and  classical quantitative characteristics related to pointwise relative compactness. We obtain new regular measures of noncompactness in the spaces taken into consideration.
The established inequalities reduce to precise formulas in some classes of subsets.
We derive Ascoli-Arzel\` a type compactness criteria.
\end{abstract}



\section{Introduction}

Let  $Y$ be a real Banach space, possibly infinite-dimensional.  Throughout the paper  we will deal  with $Y$-valued functions.
In \cite{A} Ambrosetti has extended the   Ascoli-Arzel\` a 
 theorem  to the  space of $Y$-valued  functions defined and continuous  on  a compact metric space and equipped with the supremum norm,
obtaining that a bounded subset of the space is relatively compact if and only if it is  equicontinuous   and pointwise  relatively compact, and establishing a precise  formula for the Kuratowski measure  of noncompactness of bounded and equicontinuous subsets of the space.  
Later, Nussbaum (\cite{N})  has estimated both the  Kuratowski  and the Hausdorff  measures  of noncompactness of 
 bounded subsets of  that  space, finding, as a special case, the result of Ambrosetti.
On the other hand,  Bartle \cite{Bartle} extended  the Ascoli-Arzel\` a  
 theorem  to the space 
 $\BC (\Omega, \R)$
of all real-valued functions defined,   continuous and bounded on a topological space~$\Omega$. Precisely,  a bounded subset $M$ of $\BC( \Omega, \R)$ is relatively compact if and only if  
  {for any positive $\varepsilon$ there is a finite partition  $\{A_1,  \dots, A_n \} $ of $\Omega$ such that if $x,y$ belong to the same $A_i$, then $|f(x)-f(y)| \le~\eps$, for all $f \in M$}.
In \cite{heinz} the estimates of Nussbaum have been extended 
  to the space of $Y$-valued functions defined and bounded on a  general set $\Omega$,  by means of quantitative characteristics which 
unfortunately do not allow to obtain a compactness criterion for all bounded subsets of the whole space.
 Similar  results 
 have been obtained in 
 \cite{AT}, where  the Hausdorff measure of noncompactness has been estimated in the space of totally  bounded $Y$-valued functions defined on a set  $\Omega$,
obtaining implicitly,  when $\Omega$  is a topological space,  a generalization of the Bartle criterion.
  We also mention that in \cite{BCW, BMR, BN} measures of noncompactness have been investigated  in the   space of $Y$-valued functions defined, continuous and bounded on unbounded intervals. 
   Actually,
  the compactness criterion of Ascoli-Arzel\`a  has been extended to several more general cases  which find applications in many fields of mathematical analysis.
 Of interest, in addition to the already mentioned cases, are the results that generalize the criterion to spaces of differentiable functions (among others, we recall   \cite{iraniano, Av, cianciaruso, LLWZ, SAIE}).
 In particular, in  \cite{cianciaruso}  the relative compactness has been characterized  for subsets of the space of  $Y$-valued functions defined,   $k$-times continuously  differentiable  and  bounded with all differentials up to the order $k$ on unbounded intervals. While in \cite{iraniano}   the case of real-valued functions defined   either on a compact subset of $\R^n$, or on $\R^n$ itself when   functions vanish at infinity, has been considered.
  Finally, we have to recall that  a general version  of the Ascoli-Arzel\` a    theorem concerns the characterization of relative compactness  in the space $\D(\Omega, T)$, that is, 
 the space  $\C(\Omega, T)$ of continuous functions between two topological spaces  $\Omega$ and $T$  endowed with the topology of compact convergence (see, for example, \cite[Theorem 18]{K}).
 
 The results of this paper will cover and extend the mentioned classical and more recent results on the subject.
Our first aim is to construct  regular measures of noncompactness equivalent to  the Kuratowski and the Hausdorff ones in  the    space $\B (\Omega, Y)$ of $Y$-valued   functions    defined and bounded on a nonempty set  $\Omega$,  made into a Banach space by the supremum norm. 
The main condition is  a new  equicontinuity-type concept  which we will refer to as extended equicontinuity.
 Then a~quantitative characteristic  measuring
  the degree of extended equicontinuity, together with the classical
 quantitative characteristics $\mu_\alpha,  \ \sigma_\alpha, \ \mu_\gamma, \ \sigma_\gamma$ (where $\alpha$ and $\gamma$ stand for the Kuratowski and the  Hausdorff measures, respectively) 
measuring the degree of pointwise relatively compactness  (see \cite{A,AT,heinz, N}), will allow us to estimate  the  Kuratowski and the Hausdorff measures of noncompactness of bounded subsets of the space.
When $\Omega$ is  an open subset  of a Banach space,  of independent interest are the results we are able to obtain
 in the  space $\BC^k(\Omega, Y)$ of  $Y$-valued functions defined,   $k$-times continuously  differentiable  and  bounded with all differentials up to the order $k$ on $\Omega$ and  endowed with the norm
$
\|f\|_{\BC^k}= \max\{\|f\|_\infty, \|df\|_\infty, \dots, \| d^k f\|_\infty \},
$
and also in  the complete locally convex space $\D^k(\Omega, Y)$,    that is, 
 the space  $\C^k(\Omega, Y)$   of 
 $Y$-valued functions  defined and  $k$-times continuously   differentiable on $\Omega$, endowed with the topology  of compact convergence for all differentials.  It is worth  mentioning that in   $\BC^k(\Omega, Y)$ and $\D^k(\Omega, Y)$ the formulation of our  equicontinuity-type concept,  as well as of all quantitative characteristics there considered, will   depend on  each  space in a natural way.
As a consequence,  we can formulate Ascoli-Arzel\` a type  compactness criteria 
in  spaces of $Y$-valued functions in very general settings.

The paper is organized as follows. 
    In Section~\ref{P}, we introduce some definitions and preliminary facts on measures of noncompactness. 
    Then we 
consider  in $\B (\Omega, Y)$ the generalized measure of non-equicontinuity  $\omega$ (see \cite{AT, heinz})
and the quantitative characteristics  $\mu_\alpha,  \ \sigma_\alpha, \ \mu_\gamma, \ \sigma_\gamma$. In particular, we put in evidence that  $\omega$ 
associated  with any of these quantities (which actually are equivalent), differently from what happens  in the space of totally bounded functions, does not allow  in general to characterize compactness in $\B(\Omega,Y)$.
The main results  of the paper  are presented in the following two sections. In Section~\ref{B},
we introduce our new  equiconti\-nuity-type concept,
then we obtain  inequalities and compactness criteria in any  Banach subspace of    $\B(\Omega, Y)$.
  It is worthwhile to notice that,  as   a particular case when $\Omega$ is a topological space,  the results in $\BC(\Omega,Y)$ hold when $\Omega$ is not necessarily compact and the functions not necessarily totally bounded, so we generalize at  the same time the   result of Nussbaum and the 
   Bartle criterion.
 In Section~\ref{k},   we obtain inequalities and compactness criteria  in
 Banach subspaces of $\BC^k(\Omega, Y)$    and in  the space $\D^k(\Omega, Y)$.
 In all the spaces, we always construct  regular measures of noncompactness equivalent to the  Kuratowski and the Hausdorff measures.   A precise formula for the  Kuratowski measure of noncompactness  is obtained for bounded and extendedly equicontinuous subsets of Banach subspaces  of $\B(\Omega,Y)$. Analogous results are obtained 
in   Banach subspaces  of   $\BC^k(\Omega, Y)$    and in  the space $\D^k(\Omega, Y)$.
 Further,  precise formulas for the Hausdorff measure of noncompactness are  given for bounded and  equicontinuous subsets of the spaces $\T (\Omega, Y)$  and $\TBC^k(\Omega, Y)$, consisting 
 of totally bounded functions and
 functions  of $\BC^k(\Omega, Y)$ which are compact with all differentials, respectively.  An analogous formula is obtained in  $\D^k (\Omega, Y)$.
 In  the last section, we obtain some results for pointwise relatively  compact subsets of  $\B(\Omega, Y)$ under the hypothesis that $Y$ is a Lindenstrauss space.

 In the literature  a different approach is sometimes used to obtain measures of noncompactness in some Banach spaces of $Y$-valued functions  (see, for example,  \cite{iraniano,BCW, BMR}), but not always such measures  enjoy the property of regularity.          
\section{Preliminaries} \label{P}
 In the following we will consider real linear spaces.
  Given a Banach space $E$ with zero element $\theta$,  we
denote by $B(x, r)$ the closed ball  with center  $x$ and radius $r>0$,  and   $B(E)$  will stand for  $B(\theta, 1)$. If $M$ is a
subset of $E$ we denote by  $\overline{M}$, $\mbox{co}M$  and $\overline{\mbox{co}}M$   the closure, the convex hull and the closed  convex  hull of $M$, respectively. 
 We   use the symbol $\mbox{diam}_E (M)$  for the diameter of  $M$ in $E$,  or simply $\mbox{diam} (M)$ if no confusion can arise.
 If $M$ and $N$ are subsets of $E$ and  $\lambda \in   \R$, then $M+N$ and $\lambda M$ will denote the algebraic operations on sets.
  Next, let ${\mathfrak M}_E$ be
 the family of all nonempty
bounded subsets of $E$ and let  ${\mathfrak N}_E$ be its subfamily consisting of all relatively compact sets. 
Given  a set function $\mu:  {\mathfrak M}_E \to [0, + \infty)$,  the family
$ 
\mbox{ker} \,   \mu= \{ M \in {\mathfrak M}_E \!: \ \mu (M)=0 \}$ is called
 { kernel} of $\mu$.
Following \cite{BG}, we introduce the concept of measure of noncompactness.

\begin{definition} \label{MNC}  
{\rm A set function $\mu : {\mathfrak M}_E \to [0, + \infty)$ is said to be a {\it measure of noncompactness} in $E$ if the following conditions hold for $M,N \in {\mathfrak M}_E $:
\begin{itemize}
\item[(i)] 
 $\mbox{ker} \,   \mu$ is nonempty  and 
$\mbox{ker} \,  \mu \subseteq {\mathfrak N}_E$;
\item[(ii)]  $M \subseteq N$ implies $\mu (M) \le \mu (N)$;
\item[(iii)] 
$\mu ({\overline{M}}) = \mu(M)$;
\item[(iv)] 
$\mu ({\mbox{co}}M) = \mu(M)$;
\item[(v)] 
 $ \mu ( \lambda M+ (1-\lambda) N )  \le   \lambda \mu (M) + (1-\lambda) \mu(N)$,  for $\lambda \in  [0, 1]$;
\item[(vi)] 
if  $(M_n)_n$ is a sequence of closed sets from ${\mathfrak  M}_E$
 such that $M_{n+1} \subseteq M_n$ for $ n=1,2, \dots$ and  $\lim_{n \to \infty} \mu (M_n)=0$, then the intersection  set $M_\infty = \bigcap_{n=1}^\infty M_n$ is nonempty.
\end{itemize}}
\end{definition}
We will say that  $\mu$  is a full measure if  ker$\mu= {\mathfrak  N}_E$.
A measure  $\mu$ is called sublinear if it is homogeneous and  subadditive, i.e.
 \begin{itemize}
\item[(vii)] $\mu( \lambda M) = |\lambda| \mu (M)$ for $\lambda \in \R$, and $  \mu (M+N) \le \mu(M)+ \mu(N)$, 
\end{itemize}
 moreover, $\mu$  is said to have the maximum property if
\begin{itemize}
\item[(viii)]  $\mu (M\cup N) = \max \{\mu (M),\mu (N)\}$.
 \end{itemize}
{An important class of measures of noncompactness is that constituted by regular measures,
which are full, sublinear measures with the maximum property.

We recall that given a set  $M$ in $ {\mathfrak M}_E$ 
the Kuratowski measure of noncompactness of $M$, denoted by $\alpha (M)$, is the infimum of all $\eps >0$ such that $M$ can be covered by finitely many sets of diameters  not  greater than $\eps$ 
and 
 the Hausdorff measure of noncompactness of $M$, denoted by $ \gamma_E (M)$, is the infimum of all $\eps >0$ such that $M$ has a finite $\eps$-net in $E$.
  These measures of noncompactness  are regular, besides they are equivalent,  since
$ \gamma_E (M) \le \alpha (M) \le 2 \gamma_E (M)$. 
For our pourposes, it is  also useful   to recall that the Istratescu measure of noncompactness $\beta (M)$  of $M$ is  the infimum of all $\eps >0$ such that  $M$ does not have an infinite $\eps$-separation, i.e there is no infinite set in $M$ such that  $\|x-y\| \ge \eps$ for all $x,y$ in this set, with $x \ne y$. Moreover,   the  inequalities $\beta (M) \le \alpha (M) \le 2 \beta (M)$ hold true. 
 For more details on   measures of noncompactness   the reader is referred to \cite{ADB, BG}.

  Throughout,  $ \Omega$  will be a nonempty set and $(Y, \|\cdot \|)$  a   Banach space, $\gamma$ will  always stand for~$\gamma_Y$. 
We denote by $ \mathcal{F}= \mathcal{F}(\Omega, Y)$ the  linear space of all  functions $f :   \Omega \to Y$.
Given a set of functions $ M $ in $ \mathcal{F}$,  $x \in \Omega$  and $A \subseteq \Omega$ we define 
the subsets $M(x)$ and $M(A)$ of the Banach space $Y$   by letting
\begin{equation} \label{Mdix}
 \M(x)= \{ f(x): f \in M \}, \qquad
 M(A)=  \{ f(x):  x \in A, \ f \in M \}.
 \end{equation}
The symbol  $\B= \B(\Omega, Y)$ will stand for the Banach space of all bounded functions in $ \F$, endowed with the {\bf} supremum norm
\[
\|f\|_\infty  = \sup \{  \|f(x) \|, \  x \in \Omega \}.
\]
We denote by  $\T= \T(\Omega, Y)$ the space of all $Y$-valued  functions defined and totally  bounded  on $\Omega$,  i.e.  such that $f(\Omega)$ is relatively compact,
and whenever  $\Omega$ is a topological space, we denote by  $\BC=\BC( \Omega, Y)$  the space of all $Y$-valued  functions defined, bounded and continuous on $\Omega$.
Both $\T$ and $\BC$ are Banach subspaces of $\B$.
 A function $f   \in \BC  \cap \T$  is  called  compact.

We devote the remaining part of this section to introduce and discuss  in $\B$ the quantitative characteristics,  based on the classical results on compactness given in  \cite{Bartle, F, S},  which have been  useful tools 
for the study of compactness, for example,  in spaces of totally bounded or compact functions.
Given $\M \in {\mathfrak M}_{\mathcal B}$,  we consider  (see \cite{ADP, AT, CTT, heinz, N, TTT1})
 the quantitative characteristic  
   \begin{eqnarray}  \label{omega}
 \begin{split}
\omega(M)=  \inf \{ \eps >0: &  \ \mbox{there is a finite partition} \ \{ A_1, \dots, A_n\}  \ \mbox{of} \ \Omega
\\
&\mbox{such that, for all}  \ f \in M, \  \mbox{diam} (f(A_i)) \le \eps  \  \mbox{for}   \ i= 1, \dots, n\},
\end{split}
 \end{eqnarray} 
which,  according to \cite{heinz}, generalizes  the ``measure of non-equicontinuity'' 
of Nussbaum \cite{N},
and  the quantitative characteristics (see \cite{A,heinz, N})
  \begin{equation*}  
 \mu_\alpha(\M)= \sup_{x \in \Omega} \alpha(\M(x)), \quad  \mu_\gamma(\M)= \sup_{x \in \Omega} \gamma(\M(x)).
 \end{equation*}
It is easy to verify that 
$
 \mu_\gamma (M)  \le \mu_\alpha (M)  \le 2 \mu_\gamma (M) $. A set  $\M \in {\mathfrak M}_{\mathcal B}$ is called pointwise relatively compact   if   $\mu_\alpha (M)=0$  ($\mu_\gamma (M)=0$).
 We also consider  (see \cite{AT, heinz}) the quantitative characteristics
 \[
 \sigma_\alpha(\M)=  \alpha(\M(\Omega)), \quad  \sigma_\gamma(\M)=   \gamma(\M(\Omega)).
 \]
We have
 $
 \sigma_\gamma (M) \le \sigma_\alpha (M)  \le 2 \sigma_\gamma (M) $. Moreover  $\mu_\alpha (M)  \le \sigma_ \alpha (M) $ and  $\mu_\gamma (M)  \le \sigma_ \gamma (M) $.

For the sake of completeness we recall the result of Nussbaum \cite[Theorem 1]{N}. 
\begin{theorem} \label{N}
 Let   $(\Omega, d)$ be a  compact metric space, $(\Tm,s)$ a metric space and $\C(\Omega,T)$
   the space of functions defined and continuous on $\Omega$ taking values in $T$, 
 made into a metric space    by  $d_\infty(f,g)= \sup_{x \in \Omega} d(f(x),g(x))$.
 Let  $M$ a bounded set in $\C(\Omega,T)$, set
  \[
 a=  
 \inf \{ \omega_N (\delta, M): \ \delta \ge 0 \},
 \]
 where    $\omega_N (\delta, M) = \sup \{ s(f(x), f(y)) : \ x,y \in \Omega ; \ d(x,y)  \le \delta, \ f \in M\}$ is the modulus of continuity of $M$,  then 
 \[
  \max \{ \mu_\alpha(M), \ \frac1{2}  a \} \le 
 \alpha(M) \le  \mu_\alpha (M) + 2 a.  \]
\end{theorem}
  The above theorem contains the classical result of Ambrosetti \cite[Theorem 2.3]{A}. Precisely,    if $M$ is a bounded and equicontinuous subset of  $\C(\Omega,T)$, then  $\alpha (M)= \mu_\alpha(M)$. 
  Let us observe that in \cite{heinz}}   Heinz has extended the result of Nussbaum to  bounded subsets of $\B $. In particular, in his paper (see \cite[Theorem~2 and Proposition~2~(i)]{heinz}) the following inequalities, which we state using the notations of this paper, are proved:
  \begin{equation} \label{h}
\max \left\{ \mu_\alpha (M), \ \frac1 2 \left( \omega(M)-  \sup_{ f \in M} \sigma_\alpha(\{ f \} )\right) \right\}
 \le 
 \alpha(M) \le \mu_\alpha (M) + 2 \omega(M).  
 \end{equation}
 Therefore,  if $M \in {\mathfrak M}_\T$ then  $\sup_{ f \in M} \sigma_\alpha( \{ f \} )) =0$, so that  $ \mbox{ker} (\mu_\alpha+2 \omega)=  {\mathfrak  N}_\T$
 and consequently  inequalities  (\ref{h}) furnish a criteria of compactness in  the space $\T$ of  totally bounded functions (see also \cite[Theorem~2.1]{AT}).
 But the same is not true  in the space  $\B$.  In fact,  the following Example~\ref{es elle infty} shows  that the left-hand side of (\ref{h}) can be equal to zero  for a given set $M \in {\mathfrak M}_\B$ without being $\alpha(M)=0$, which means that  the left-hand side of  (\ref{h})  vanishes on   sets $M \notin {\mathfrak N}_\B$. While  Example~\ref{EX1} shows that 
there are sets  $M  \in {\mathfrak N}_\B$ such that the right-hand side of (\ref{h}) does not vanish on $M$.
In other words, $\mbox{ker} (\mu_\alpha + 2 \omega) = {\mathfrak N}_\T$ and it is  a proper subset of  ${\mathfrak N}_\B$,
so that    the relations given in  (\ref{h}) are not adequate to characterize compactness in $\B$. 

 Before giving the examples, let us observe that given  a function $f \in  \B$, we have   $\alpha (\{f\})=  \gamma_\B (\{f\})=0$  and also $\mu_\alpha  (\{f\}) = \mu_\gamma  (\{f\}) =0$, due to the fact that  $\alpha (\{f (x)\})= \gamma (\{f(x) \} )=0$ for all $x \in \Omega$.
Moreover, for $f \in \B$ we have
  \begin{equation} \label{DUE}
  \omega (\{f\}) = \sigma_\alpha(\{f\}).
  \end{equation}
  Indeed, given $a > \sigma_\alpha (\{f\})$, that is $a> \alpha (f(\Omega))$, and  $\{ Y_1, \dots, Y_n\}$    a partition of $f(\Omega)$ with 
  $\mbox{diam}
(Y_i) \le a$ for $i = 1,...,n$, 
  considering the finite partition $\{A_1, \dots, A_n \}$ of $\Omega$ with $A_i = f^{-1} (Y_i)$,     we find $\omega(\{ f \}) \le a $.
Vice versa,  given  $b > \omega(\{ f \})$ and $\{A_1,...,A_n \}$  a finite partition of~$\Omega$  with $\mbox{diam}(f(A_i)) \le b$ for $i = 1,...,n$, we have $f(\Omega) =  \bigcup_{i=1}^n f(A_i)$, which implies $ \alpha (f(\Omega)) \le b$. Thus     (\ref{DUE}) follows.
\begin{example} \label{es elle infty}
{\rm
  Let $\Omega = [0, + \infty)$ and  $Y = \ell_\infty$. 
  Consider
the  sequence $(f_k )_k$  in   $\B([0, + \infty), \ell_\infty)$ defined by
\[
f_k(x)= \begin{cases}
e_n  \quad \mbox{for}  \ \ x= k - \frac{1}{n} , \ \ \mbox{for} \ \ n=1,2, \dots 
\\ 
  \theta \quad \mbox{otherwise} .
\end{cases}
\]
Set $M= \{ f_k:  \ k=1,2,\dots \}$.
Given $x \in [0, + \infty)$, $M(x)=\{ \theta \}$ if $x \notin  \bigcup_{k=1}^{ \infty} \left\{ n - \frac 1k : \ n=1,2, \dots   \right\} $ or $M(x) = \{\theta,e_n \} $  if $x  \in 
 \bigcup_{k=1}^{ \infty} \left\{ n - \frac 1k : \ n=1,2, \dots   \right\}$,  so that $M$ is pointwise relatively compact.
 Using (\ref{DUE})  and taking into account that    $\mbox{\rm diam} \left(  \{ e_n: n=1, 2, \dots \} \right) =1 $   we deduce,  for each $k \in \N$,
\[
   \omega (\{f_k \}) =  \alpha \left( f_k (  [0, + \infty)) \right) =      \alpha \left(   \{ e_n: n=1, 2, \dots \} \right) \le 1.
 \]
Since the Istratescu measure  $\beta \left(   \{ e_n: n=1, 2, \dots \} \right) =1$,  we get  
$\alpha \left(   \{ e_n: n=1, 2, \dots \} \right) =1$ and hence $ \  \omega (\{f_k \}) =1$ which implies $\omega (M) \ge 1$.
On the other hand, since  $\mbox{diam} (f_k([0, + \infty))= \mbox{diam}   \left(   \{ e_n: n=1, 2, \dots \} \right) = 1$ for all $k$, taking the partition $\{[0, + \infty) \} $   we obtain $\omega (M) \le 1$, thus $\omega (M)=1$.
Now we prove $\alpha (M)=1$. First we notice that for $k \ne s$ we have 
\[
 \|f_k - f_s \|_\infty
=    \sup_{x \in \Omega}\|f_k(x) - f_s (x)\|=1,
\]
so that  $\beta (M) \ge 1$ and hence
$
  \alpha (M) \ge 1 .
$
  At the same time $   \mbox{diam}(M)= \sup_{k,s \in\N}  \|f_k-f_s \|_\infty    =1             $,
so that $\alpha (M) \le 1$, and our assert follows.
\\
Then, on the one hand $\alpha(M)=1$, on the other hand
\[
\mu_\alpha (M)=0,   \quad \omega (M)= 1= \sup_{ k \in \N} \sigma_\alpha(\{ f_k \} ),
\]
which in turn give $M \notin {\mathfrak N}_{\B([0, + \infty), \ell_\infty)}$ and
\[
\max \left\{ \mu_\alpha (M), \ \frac1 2 \left( \omega(M)-  \sup_{ k \in \N} \sigma_\alpha(\{ f_k \} )\right) \right\}=0.
\]
}
\end{example}
 \begin{example} \label{EX1}
 {\rm 
 Assume the Banach space $Y$ to be  infinite-dimensional. Let  $\Omega= B(Y)$  and  $I$  be  the identity function on $B(Y)$. Choose $M= \{I\}$, then
 \[
  \alpha (M)=  \gamma_\B (M)=\mu_\alpha (M)= \mu_\gamma (M)=0, \quad \sigma_\alpha (M)=\omega (M)=2, \quad \sigma_\gamma (M)= 1.
 \]
 In particular,   $M \in {\mathfrak N}_{\B(B(Y), Y)}$, but  the right-hand side of (\ref{h}) does not vanish on $M$.
}
\end{example}
   Further, in view  of the above example, the quantitative characteristics $\omega$,  $\sigma_\alpha$ and $\sigma_\gamma$   cannot be used in general  to characterize compactness in $\B$.  The goal of the next section  is to use a new  equicontinuity-type concept  and a quantitative characteristic modeled on it which will allow  us  to fill in this gap.

  To close this section, given    $\psi : {\mathfrak M}_\X \to [0, \infty)$  any  set function in $\{\mu_\alpha,  \ \sigma_\alpha, \ \mu_\gamma, \ \sigma_\gamma\}$,  from the properties of $\alpha$    and $\gamma$ we derive  that   $\psi $ satisfies axioms $(ii)\mbox{-}(v)$ of Definition~\ref{MNC}. 
  \begin{proposition} \label{(ii)-(v)} Assume  $ \psi \in \{\mu_\alpha,  \ \sigma_\alpha, \ \mu_\gamma, \ \sigma_\gamma\}$.  Let $M, N \in {\mathfrak M }_\X$ and $\lambda \in [0,1]$. Then
 \begin{itemize}
 \item[$(ii)$] $M \subseteq N$ implies $\psi (M) \le \psi (N)$; 
  \item[$(iii)$]  
  $\psi (\overline{M}) = \psi (M)$; 
   \item[$(iv)$] $\psi  (\mbox{co}{(M)}) = \psi (M)$; 
    \item[$(v)$] $ \psi (\lambda M+ (1-\lambda) N) \le  \lambda \psi (M) + (1-\lambda) \psi (N)$.
     \end{itemize}
 \end{proposition} 
\begin{proof}
 As the proof works  in the same way for $\alpha$ and $\gamma$, we 
assume $\psi \in \{ \mu_\alpha, \  \sigma_\alpha \}$.  
Let $M, N \in {\mathfrak M }_\X$ and $\lambda \in [0,1]$.  If $M \subseteq N$,  property $(ii)$ follows immediately  from  the definition of $\psi$.

(iii) Using (ii) we have  $\psi (M) \le \psi  ( \overline{M})$. We prove the  converse inequality.
\\
Let $\psi= \mu_\alpha$.  Let $x \in \Omega$, $y \in \overline{M}(x)$ and $f \in \overline{M}$ such that $y=f(x)$.
Let  $(f_n)_n$ be a sequence  of functions in $M$ such that $\|f_n-f\|_\infty \to 0$ as $n \to \infty$, then
$\|f_n(x)-f(x)\| \to 0$. Hence $f(x)=y \in \overline{M(x)}$, so that $\overline{M}(x) \subseteq\overline{M(x)}$ and
$
 \gamma(\overline{M}(x))  \le \gamma(\overline{M(x)}).  
$
 Since $ \alpha(\overline{M(x)})  = \alpha(M(x))$, we obtain  $ \alpha(\overline{M}(x)) \le \alpha(M(x)) $, for all $x \in \Omega$, which implies $ \mu_\alpha(\overline{M}) \le \mu_\alpha(M) $.
 \\
Let $\psi= \sigma_\alpha$.  
Given $x \in \Omega$, $y \in \overline{M}(\Omega)$ and $f \in \overline{M}$ such that $y=f(x)$. Repeating the same argument  as before, we find
  $\overline{M}(\Omega) \subseteq \overline{M(\Omega)}$ and then   $\sigma_\alpha(\overline{M} ) \le \sigma_\alpha(M)$.

  (iv)
Using (ii) we have  $\psi (M) \le \psi  ( \mbox{co}M)$. We prove the converse inequality.
\\
Let $\psi= \mu_\alpha$. 
Let $x \in \Omega$, $y \in (\mbox{co}M)(x)$ and $f \in \mbox{co}M$ such that $y=f(x)$.
Fix $f_1, \dots, f_n \in M$ and $\lambda_1, \dots, \lambda_n \in [0,1]$ with $\sum_{i=1}^n \lambda_i =1$ such that $f= \sum_{i=1}^n \lambda_i f_i$.
Then 
$y= f(x) \in \mbox{co}M(x)$. Therefore $ ( \mbox{co}M)(x)  \subseteq \mbox{co}M(x)$ for all $x \in \Omega$, and from this  follows $ \sigma_ \alpha(\mbox{co}M) \le  \sigma_ \alpha (M)$, as desired. 
\\
If $\psi= \sigma_\alpha$, 
the converse inequality follows from the fact that one proves $ (\mbox{co}M)(\Omega)  \subseteq \mbox{co}M(\Omega)$.

(v)  It is immediate, in one case from
$
 (\lambda M+ (1-\lambda) N)(x) \subseteq   (\lambda M)(x) + ((1-\lambda) N)(x)$ for all  $ x \in \Omega$,
and in the other, from
$
 (\lambda M+ (1-\lambda) N)(\Omega) \subseteq   (\lambda M)(\Omega) + ((1-\lambda) N)(\Omega) .
$
 \end{proof}
\begin{remark}
{\rm
   From   \cite[Example 1]{N} we see that  the  set functions  $ \mu_\alpha, \ \sigma_\alpha, \  \mu_\gamma$ and  $\sigma_\gamma$  are not in general    measures of noncompactness.
Let  $E= \C([0,1], \R)$ and  $M= \{ f_n: n \ge 3 \}  \in {\mathfrak M}_E$, where     \begin{eqnarray*} 
f_n(t)= \left\{
\begin{array}{lll}
0  \quad  &\mbox{for} \  \  0 \le t <  \frac1{2} - \frac1{n},
\\
[.1in]
(t-  \frac1{2} + \frac1{n}) \frac{n}{2 } \quad  \ & \mbox{for} \  \  \frac1{2} - \frac1{n}  \le  t  <   \frac1{2}+ \frac1{n},
\\
[.1in]
 1
  &\mbox{for} \  \  \frac1{2} + \frac1{n} \le t \le  1.
\end{array}
\right.
\end{eqnarray*}
Then    $M \notin {\mathfrak N}_E$ and $\mu_\gamma (M)=\mu_\alpha (M)=\sigma_\gamma (M)=\sigma_\alpha (M)=0$,   so that,  none of  the given set functions  satisfies  condition    (i) of Definition~\ref{MNC}.  
}
\end{remark}

 \section{Compactness in Banach subspaces of  $\B(\Omega, Y)$} \label{B}
Throughout this section,  $\X$   will stand for  a  Banach subspace of~$\B $, possibly $\B$ itself,  endowed  with $\| \cdot \|_\infty$.   Let us introduce the following equicontinuity-type concept in our general setting. 
  \begin{definition} 
 {\rm  
 We say that  a set $M \in {\mathfrak M}_\X$ is 
 {\it extendedly equicontinuous}
  if
for any $  \eps >0$ there  are a finite partition       $ \{ A_1, \dots, A_n\} $ of  $\Omega$
and a finite set $ \{ \varphi_1, \dots, \varphi_m \} $ of functions in $\X$  such that, for all $ f \in M$, 
there is $ j \in \{1, \dots, m\}$  with $ \mbox{diam}( (f - \varphi_j) (A_i) ) \le \eps$  for $ i= 1, \dots, n$.  
}
  \end{definition}
 Next, for a set $M \in {\mathfrak M}_\X$ the introduce the new quantitative characteristic  $\omega_\X (M)$ that  will measure the degree of extended equicontinuity of  $M$. 
 \begin{definition} \label{parametronuovo}
{\rm   We define the set function $\omega_\X: {\mathfrak M}_\X \to [0, + \infty)$ by setting
 \begin{eqnarray*} 
  \begin{split}
\omega_\X (M) = \inf \{ \eps >0: \ & \mbox{there  are a finite partition} \ \{ A_1, \dots, A_n\}  \ \mbox{of} \ \Omega
\\
&  \mbox{and a finite set} \ \{ \varphi_1, \dots, \varphi_m \} \ \mbox{in} \  \X  \ \mbox{such that, for all}  \ f \in M, 
\\
&  \mbox{there is} \ j \in \{1, \dots, m\} \ \mbox{with} \  \mbox{diam}( (f - \varphi_j) (A_i) ) \le \eps \ \mbox{for} \ i= 1, \dots, n\}.
 \end{split}
 \end{eqnarray*} 
 }
 \end{definition}
 Clearly a set $M$ is   extendedly equicontinuous if and only if $\omega_\X (M)= 0$. 
 \\
Let us notice that, for $M \in {\mathfrak M}_\X$, we have
$
 \omega_\B (M) \le \omega_\X (M) \le 2 \omega_\B (M).
$
Indeed, the left inequality is immediate. To show the right one, let $a > \omega_\B (M)$, $\{A_1, \dots, A_n\}$ a finite partition of $\Omega$
and $\{ \varphi_1, \dots, \varphi_m \}$ a finite set in  $\B$  such that, for all $f \in M$
there is $ j \in \{1, \dots, m\}$  with $\mbox{diam}( (f - \varphi_j) (A_i) ) \le  a$ 
for $i= 1, \dots, n$.
Next   for  each $j$, set 
 $ 
 M_j= \{ f \in M :   \mbox{diam}( (f - \varphi_j) (A_i) )  \le a \}
 $
 and choose $\psi_j$ arbitrarily in $M_j$. Then, $\{\psi_1, \dots, \psi_m\}$ is a finite set in $\X$  
  such that, for all $f \in M$
there is $ j \in \{1, \dots, m\}$  with
$  \mbox{diam}( (f - \psi_j) (A_i) ) \le  \mbox{diam}( (f - \varphi_j) (A_i) ) +  \mbox{diam}( (\psi_j- \varphi_j) (A_i) )  \le 2a$  for $i= 1, \dots, n$, as expected.
\begin{proposition} \label{propTB}						
If $\X \subseteq \T$ and  $M \in {\mathfrak M}_\X$, then 
 \[
 \omega(M)= \omega_\X (M).\]
\end{proposition}
 \begin{proof}
 Take $\{\varphi_0 \}$, where  $\varphi_0$ denotes the null function in $\X$, as a finite set of functions in the definition of $\omega_\X$, then $\omega_\X (M) \le \omega(M)$.
Now, we prove the reverse inequality.
Let $a > \omega_\X (M)$,     let  $\{A_1, \dots, A_n\}$ be  a finite partition of $\Omega$
and $\{ \varphi_1, \dots, \varphi_m \}$ a finite set in  $\T$  such that, for all $f \in M$
there is $ j \in \{1, \dots, m\}$  such that $\mbox{diam}( (f - \varphi_j) (A_i) ) \le a$ 
for $i= 1, \dots, n$.
Set $T= \bigcup_{j=1}^m \varphi_j (\Omega)$. Since $\varphi_j \in \T$ for $j=1, \dots, m$ we have $\gamma(T)=0$.
Therefore, for $\delta >0$ arbitrarily fixed, there are $y_1, \dots, y_k \in Y$ such that  
$T \subseteq \bigcup_{l=1}^k B(y_l, \delta)$.  Hence $T= \bigcup_{l=1}^k K_l$, where
$K_l= T \cap B(y_l, \delta)$, for $l=1, \dots, k$.
Set $B_l= K_l \setminus (\bigcup_{s=0}
^{\ell-1} K_s) $, for $l=1, \dots,k$, where $K_0= \emptyset$.
Then $\{B_1, \dots, B_k \}$ is a finite partition of $T$, so that for all $j \in \{1, \dots, m \}$ the family 
$\{ \varphi_j^{-1} (B_1), \dots, \varphi_j^{-1} (B_k) \}$ is a finite partition of $\Omega$.
Let $\{S_1, \dots, S_q \}$ be the partition of $\Omega$ generated by the partitions
$\{ \varphi_j^{-1} (B_1), \dots, \varphi_j^{-1} (B_k) \}$ ($j=1, \dots, m $) and by the partition  
 $\{A_1, \dots, A_n\}$.
Therefore,  for all $f \in M$ and for all $r  \in \{ 1, \dots, q  \}$
we  can choose $j \in \{1, \dots, m \}$ such that
 \begin{eqnarray*}
\mbox{diam} ( f(S_r)	)&=& \sup_{x,y \in S_r}  \| f(x)- f(y) \| 
\\
&\le& 
 \sup_{x,y \in S_r}    \| (f- \varphi_j)(x)- (f- \varphi_j)(y) \| +  \sup_{x,y \in S_r}\| \varphi_j(x)- \varphi_j(y) \| 
 \le a + 2 \delta,
\end{eqnarray*}
so that $\omega(M) \le \omega_\X (M)$. 
  \end{proof}
  The set  function $\omega_\X: {\mathfrak M}_\X \to [0, + \infty)$ satisfies axioms (ii)-(v) of Definition~\ref{MNC}. Clearly,  it is not in general  a measure of noncompactness in~$\X$.
 \begin{proposition} \label{(ii)-(v)bis}  Let $M, N \in {\mathfrak M }_\X$ and $\lambda \in [0,1]$. Then 
  \begin{itemize}
 \item[(ii)] $M \subseteq N$ implies $\omega_\X (M) \le \omega_\X(N)$;
  \item[(iii)]  $\omega_\X (\overline{M})=  \omega_\X({M})$;
   \item[(iv)] $\omega_\X (\rm{co}M)= \omega_\X(M)$;
    \item[(v)] $ \omega_\X (\lambda M+ (1-\lambda) N) \le  \lambda \omega_\X (M) + (1-\lambda) \omega_\X (N)$.
 \end{itemize}
 \end{proposition}
\begin{proof}
Throughout this proof, given $M \in {\mathfrak M}_\X$,    $a > \omega_\X (M)$  
we denote by $\{A_1, \dots, A_n\}$ a finite partition of $\Omega$
and  by $\{ \varphi_1, \dots, \varphi_m \}$ a finite set in  $\X$  
 such that for all $ f\in M$ there is  $ j \in \{1, \dots, m\}$  such that 
$
\mbox{diam}( (f - \varphi_j) (A_i) ) \le a$ for $i=1, \dots,n$.
\\
Property  (ii)  follows immediately  from the definition of $\omega_\X$.

Next, observe that the inequality  $\omega_\X ({M}) \le \omega_\X (\overline{M})$ follows from (ii).
To prove the converse inequality, let $\delta>0$ be arbitrarily fixed and let $g \in \overline{M}$.
 Choose $f \in M$ and 
 $ j \in \{1, \dots, m\}$ such that $\| f-g\|_\infty \le \delta$ 
 and
  $\mbox{diam}( (f - \varphi_j) (A_i) ) \le a$  for $i= 1, \dots, n$.
  Thus, for all $i$,  we have
 \begin{eqnarray*}
  \mbox{diam}( (g - \varphi_j) (A_i) ) &=& \sup_{x,y \in A_i} \| (g- \varphi_j)(x) -  (g- \varphi_j)(y) \| 
\\
& \le &
   \sup_{x,y \in A_i} \left( \| (f- \varphi_j)(x) -  (f- \varphi_j)(y) \| 
   +
    \| g(x)-f(x) \|  
    +
     \| g(y) - f(y) \|    \right) 
     \\
     &  \le &  a +  2 \delta.
   \end{eqnarray*}
   Hence, by the arbitrariness of $\delta$, it follows
     $\mbox{diam}( (g - \varphi_j) (A_i) ) \le a$, and so, by the arbitrariness of~$a$, we obtain
      $\omega_\X(\overline{M}) \le  \omega_\X({M})$.  We have proved (iii).
   
To prove (iv) enough to show $\omega_\X (\rm{co}M)\le \omega_\X(M)$.
  Let $g \in \mbox{co}M$ be arbitrarily fixed. Let $f_1, \dots, f_k \in M$ and $\lambda_1, \dots,  \lambda_k \in [0,1]$ with $\sum_{s=1}^k \lambda_s=1$ such that $g=\sum_{s=1}^k \lambda_s f_s$,
  Let denote $H$ the set of all functions $i \to h(i)$ of $\{1, \dots  k\}$ into $\{1, \dots, m\}$.
 Fix $h \in H$ such that, for all $s \in \{ 1, \dots, k\}$, we have
 \[
 \mbox{diam}( (f_s - \varphi_{h(s)}) (A_i) ) \le a \quad \mbox{for} \ i= 1, \dots, n.
 \]
 We observe that $\mbox{co}\{ \varphi_1, \dots, \varphi_m \}$ is a compact subset of $\X$. Hence, given $\delta>0$, we can choose a finite $\|\cdot\|_\infty$$-\delta$-net $\{ \psi_1, \dots, \psi_l \}$ for 
 $\mbox{co}(\{ \varphi_1, \dots, \varphi_m \})$ in $\X$.
 Then the  function 
$
 \sum_{s=1}^k \lambda_s \varphi_{h(s)}$ belongs to $\mbox{co}(\{ \varphi_1, \dots, \varphi_m \})$,  
 so that there is $v \in \{1, \dots, l\}$ such that 
 \[
 \left\|   \sum_{s=1}^k \lambda_s \varphi_{h(s)}  - \psi_v \right\|_\infty \le \delta.
 \]
 Therefore,  for all $i \in \{1, \dots,n \}$, we obtain
 \begin{eqnarray*}
 \mbox{diam} \left((g- \psi_v)(A_i) \right)&=&
  \sup_{x,y \in A_i}  \left\|    \left(       \sum_{s=1}^k \lambda_s f_s - \psi_v\right)(x) - \left(\sum_{s=1}^k \lambda_s f_s- \psi_  v\right)(y) \right\|
 \\
   & \le &  \sup_{x,y \in A_i}  \Big(  \left\|   \sum_{s=1}^k \lambda_s \left( f_s- \varphi_{h(s)} \right)(x)     
   -  \sum_{s=1}^k \lambda_s \left( f_s- \varphi_{h(s)} \right)(y)   
    \right\| 
    \\
  &  +&   
      \left\|   \sum_{s=1}^k \lambda_s  \varphi_{h(s)}(x) - \psi_v(x) \right\|
      +
     \left\|   \sum_{s=1}^k \lambda_s  \varphi_{h(s)}(y) - \psi_v(y) \right\|   \Big) 
            \le a + 2 \delta.
  \end{eqnarray*}
 By the arbitrariness of $\delta$ and $a$ we have the desired result. So the proof of (iv) is complete.
  
 Finally, we prove (v).
 Given $  N \in {\mathfrak M}_\X$,     $b> \omega_\X(N)$, without loss of generality we may still assume that $\{A_1, \dots, A_n\}$ is a finite partition of $\Omega$
and    $\{ \varphi_1, \dots, \varphi_m \}$ is a finite set in  $\X$  
 such that for all   $g \in N$
there is $  k \in \{1, \dots, m\}$  such that 
$\mbox{diam}( (g - \varphi_k) (A_i) ) \le b$, for $ i= 1, \dots, n.
$
\\
    Let $\lambda \in [0,1]$ and $\z \in  \lambda M+ (1-\lambda) N$,  then  $\z=  \lambda f+ (1-\lambda) g$  \ with $ f \in M$ and $g \in N$. 
   Choose $j,k \in \{ 1, \dots,m \}$ such that
$
     \mbox{diam}( (f - \varphi_j)(A_i) ) \le a  \ \mbox{and}  \ \mbox{diam}( (g- \varphi_k)(A_i) ) \le b
$
     for all $i \in   \{ 1, \dots, n \}$.
     Observe that $\mbox{co} \{ \varphi_1, \dots, \varphi_m \}$ is a compact  set in $\X$. Thus, given $\delta>0$, we can choose a finite $\|\cdot\|_\infty \mbox{-} \delta$-net $\{ \psi_1, \dots, \psi_l \}$ for 
 $\mbox{co}\{ \varphi_1, \dots, \varphi_m \}$ in $\X$.
\\
 Next, let
$s \in \{1, \dots, l\}$ such that $\| \psi_s-  ( \lambda \varphi_j+ (1-\lambda) \varphi_k)\|_\infty \le \delta$. 
We have, for $i \in \{1, \dots,n \}$, 
\begin{eqnarray*}
\mbox{diam}( ( \z &-& \psi_s) (A_i) )  = \sup_{x,y \in A_i} \| ( \z - \psi_s)(x)- ( \z - \psi_s)(y) \|
\\
&=& \sup_{x,y \in A_i} \| (\lambda f+ (1-\lambda)g- \psi_s)(x)- ( \lambda f+ (1-\lambda)g - \psi_s)(y) \|
\\
&\le &
\sup_{x,y \in A_i} \Big(   
\|  \psi_s(x)-  ( \lambda \varphi_j + (1-\lambda) \varphi_k)(x) \|
+ \|  \psi_s(y)-  ( \lambda \varphi_j + (1-\lambda) \varphi_k)(y) \|
\\
&+& \|    \lambda (f - \varphi_j)(x)    -     \lambda (f- \varphi_j)(y) \|
 + \| (1-\lambda) (g- \varphi_k)(x)-  (1-\lambda) (g- \varphi_k)(y) \|
\Big)   
\\
& \le & 2 \delta+ \lambda a + (1-\lambda)b.
\end{eqnarray*}
In virtue of the arbitrariness of $\delta, \ a$ and  $b$  we obtain (v).
\end{proof}
\subsection{MNC equivalent to $\alpha$} 
In this subsection, we prove  equivalent relations in Banach subspaces $\X$ of $\B$ for the Kuratowski measure of noncompactess.
 We establish a criteria of compactness  and a precise formula for the Kuratowski measure of bounded extendedly equicontinuous subsets of the space. The results are new when $\X$ properly contains $\T$.
  \begin{theorem} \label{maintheoremalpha}
 Let  $\X \subseteq \B$ and let   $M \in {\mathfrak M}_\X$. Then
 \begin{equation} \label{estimate2n}
  \max \{ \mu_\alpha(M), \ \frac1{2} \omega_\X (M) \} \le 
 \alpha(M) \le \mu_\alpha (M) + 2\omega_\X (M).
 \end{equation}
 \end{theorem}
 \begin{proof}
 We prove the left inequality.
 Let $a > \alpha (M)$ and let $\{M_1, \dots, M_n\}$ be a finite cover of $M$ such that $\mbox{diam}(M_i) \le a$ for $i= 1, \dots,n$.
 For all $x \in \Omega$, we have
 $\M(x) \subseteq \bigcup_{i=1}^n M_i(x)$. Moreover, for all $i\in \{1,\dots , n \}$, we have  
 \[
 \mbox{diam}(M_i(x)) = \sup_{f,g \in M_i} \| f(x)-g(x) \| \le 
  \sup_{f,g \in M_i} \| f-g \|_\infty = \mbox{diam}(M_i) \le a.
 \]
 Hence
 $\alpha (M(x)) \le a$ for all $x \in \Omega$, i.e. $\mu_\alpha (M) \le a$.
By the arbitrariness of $a$, we have 
\begin{equation} \label{left3.1}
\mu_\alpha (M) \le \alpha(M).
\end{equation}
 Now, choose $\{\Omega \}$ as a partition of $\Omega$ and $\{\varphi_1, \dots, \varphi_n\}$ as a finite set in $\X$, taking  $\varphi_i \in M_i$ for $i=1, \dots, n$.
 Let $f \in M$, fix $i \in \{1, \dots, n \}$ such that $\|f - \varphi_i\|_\infty \le a$,
 then
 \[
 \mbox{diam}   ((f- \varphi_i)(\Omega)) = 
 \sup_{x,y \in \Omega} \| (f- \varphi_i)(x)-(f- \varphi_i)(y) \|
 \le 2  \| f- \varphi_i\|_\infty \le 2a .
 \]
 Thus we derive that  $\omega_\X (M) \le 2a$, and, by the arbitrariness of $a$,
 \begin{equation} \label{left3.2}
 \omega_\X (M) \le 2 \alpha(M).
 \end{equation}
Then (\ref{left3.1})   and  (\ref{left3.2})  give $\max \{ \mu_\alpha(M), \ \frac1{2} \omega_\X (M) \} \le \alpha (M)$.

Next, we prove the right inequality. 
First let $a > \omega_\X (M)$, 
let $ \{ A_1, \dots, A_n\}$   be a finite partition of  $\Omega$ and $\{ \varphi_1, \dots, \varphi_m \}$
  a finite set in $\X$  such that, for all $ \ f \in M$,  
there is $j \in \{1, \dots, m\}$ such that 
$
 \mbox{diam}( (f - \varphi_j) (A_i) ) \le a \ \mbox{for} \ i= 1, \dots, n.
$
  Set,  for $j=1, \dots , m$,
  $
  M_j= \{ f \in M: \ \mbox{diam}((f- \varphi_j)(A_i)) \le a \ \mbox{for} \ i=1, \dots , n \}$. Fix $x_i \in A_i$ for $i=1, \dots,n$,  observe that for each $j$ we have
\[
\alpha (\bigcup_{i=1}^n M_j(x_i))  = \max_{i=1}^n \alpha (M_j(x_i))  \le   \max_{i=1}^n \alpha (M(x_i)) \le \mu_\alpha  (M).
\]
Then, let $b >  \mu_\alpha(M)$ and, for any fixed $j$,  let $\{ B_1^{j}, \dots, B_{l_j }^{j}\}$ be a finite cover of  \ $\bigcup_{i=1}^n M_j (x_i)$ such that $\mbox{diam}   (B_s^{j}) \le b$ for $s =1, \dots,  l_j$.
Let  $j$ be fixed. Denote by   $H^j$  the set of all functions $h^j: \{1, \dots n\} \to \{ 1, \dots,  l_j\}$, and for $h^j \in H^j$ consider  the set 
 \[
 M_{j,h^j}=\{ f \in M_j :  \   f(x_i) \in B^j_{h^j(i)}, \ \mbox{for} \ i=1,\dots,n \}.
 \]
 Then, $\{ M_{j, h^j}: \ h^j \in H^j  \}$ is a finite cover of $M_j$,  and $\mbox{diam} (M_{j, h^j}) \le b +2a$ for all $j$.
 \\
 In fact, for $f,g \in M_{j,h^j}$,
 we have
 \begin{eqnarray*}
 \| f-g \|_\infty &=& \max_{i=1}^n \sup_{x \in A_i}  \|f(x)-g(x)\|
 \\
 &\le&  \max_{i=1}^n \sup_{x \in A_i}   ( \|f(x_i)-g(x_i ) \| +   \| (f- \varphi_j )(x)- (f- \varphi_j)(x_i) \|  
 \\ &  + &  
  \| (g- \varphi_j)(x)- (g- \varphi_j)(x_i) \|   )
 \le b +2a.
 \end{eqnarray*} 
Since $\{M_j: j=1, \dots,m \}$   is a finite cover of $M$, we infer $ \alpha (M) \le b +2a$, and by the arbitrariness of $a$ and $b$ we get  $\alpha (M) \le \mu_\alpha (M) + 2 \omega_\X (M)$.
 \end{proof}
 
 Now, from the inequalities we have proved, we obtain the following   compactness criterion in~$\X$. 
 \begin{corollary} \label{cor1}
Let $\X \subseteq \B$. A subset $M$ of $\X$ is relatively compact if and only if it is bounded, extendedly equicontinuous and   pointwise relatively compact.
\end{corollary}
  The inequalities (\ref{estimate2n})  reduce, for a class of subsets,  to a precise  formula of Ambrosetti-type.
 \begin{corollary} \label{Ambrosetti}
 Let $\X \subseteq \B$ and assume that  $M \in  {\mathfrak M}_\X $  is extendedly equicontinuous.
 Then  
 \[
 \alpha (M)= \mu_\alpha(M).
 \] 
 \end{corollary} 
 In the general case, we  get a regular measure of noncompactness equivalent to that of  Kuratowski.
  \begin{corollary}  \label{newMNC}   Given  $\X \subseteq \B$, the set function $ \mu_\alpha+ 2 \omega_{\X}:  {\mathfrak M}_\X \to[0, +\infty)$ is a  regular measure of noncompactness in $\X$ equivalent to the Kuratowski measure $\alpha$.
\end{corollary}
\begin{proof}
From Proposition~\ref{(ii)-(v)}    and Proposition~\ref{(ii)-(v)bis} it follows that  $ \mu_\alpha+ 2 \omega_{\X}$
  satisfies conditions (ii)-(v) of
 Definition~\ref{MNC}. 
Further conditions  (i) and (vi) 
are consequences  of  Corollary~\ref{cor1}.  
 Finally, it can be easily verified that  both $ \mu_\alpha$ and $\omega_{\X}$  are  sublinear and enjoy the maximum property so the same is true for $ \mu_\alpha+ 2 \omega_{\X}$.
  The equivalence follows from Theorem~\ref{maintheoremalpha}.
\end{proof}

We end this subsection looking at Banach subspaces of $\T$. In such a case, the quantitative characteristic $\sigma_\alpha$ can be used to improve the left-hand side  of (\ref{estimate2n}). We need the following lemma.

\begin{lemma} \label{lemma1}      
Let  $\X \subseteq \T$ and let $M \in {\mathfrak M}_\X$. Then
  \[
  \sigma_\alpha (M) \le \alpha (M).
  \]
  \end{lemma}
\begin{proof}
Let    $M \in {\mathfrak M}_\X$, $a > \alpha (M)$ and let $\{M_1, \dots,M_n \}$  be  a finite cover of $M$ with $\mbox{diam} M_i~\le~a $.
Take  $f_i \in M_i$,  for  $i=1, \dots n$.  Since each $f_i$ is totally bounded we have that $\alpha(\bigcup_{i=1}^n f_i(\Omega))~=~0$. Thus, given $ \eps >0$  arbitrarily fixed, we    choose  a finite cover  $\{B_1, \dots, B_m\}$    of 
$\bigcup_{i=1}^n f_i(\Omega)$ such that $\mbox{diam} (B_j) \le  \eps$, for $j=1,\dots,m$.
Next, for each $j$, fix $y_j \in B_j$.
Now let $f \in M$ and $x \in \Omega$ be arbitrarily fixed.
First choose $i \in \{1,\dots,n \}$ such that $\|f-f_i \|_\infty \le  a$, then $j \in \{ 1, \dots,m\}$ such that $f_i(x) \in B_j$, that is $\|f_i(x)- y_j \|_\infty \le  \eps$. Then
$
\|f(x) - y_j \|_\infty \le a + \eps ,
$
and by the arbitrariness of $\eps$ we have
$
\|f(x) - y_j \|_\infty \le  a,
$
so that  $\{B(y_1, a), \dots B(y_m, a) \}$ is a finite cover of
 $M(\Omega) $.  As $\sigma_\alpha(M)= \alpha (M(\Omega))$, the proof is complete.
\end{proof}
  \begin{theorem} \label{coralpha}
If $\X \subseteq \T$ and  $M \in {\mathfrak M}_\X$,
  then
  \[
  \max \{ \sigma_\alpha(M), \ \frac1{2} \omega (M) \} \le 
 \alpha(M) \le \mu_\alpha (M) + 2\omega (M).  \]
  \end{theorem}
\begin{proof}
  From Lemma \ref{lemma1} and Proposition \ref{propTB} we have $\sigma_\alpha (M) \le \alpha (M)$ and $\omega_\X (M)= \omega (M)$. Hence, Theorem \ref{maintheoremalpha} gives the thesis.
\end{proof}
From the above result we get  that
 $\omega(M)=0$ implies $ \alpha (M)= \mu_\alpha (M)=  \sigma_\alpha (M)$. This  extends   \cite[Lemma 2.2]{A} from the case of sets of $Y$-valued functions defined and continuous on a compact metric space to the case of sets of $Y$-valued functions defined and  totally bounded on a general set~$\Omega$.

\subsection{MNC equivalent to $\gamma$} 
 
 We now provide similar estimates for the Hausdorff measure of noncompactness,   more accurate than those one can derive  from the previous results using the known equivalence between measures of noncompactness and   the involved quantitative characteristics.
   At first, the  following lemma gives   the upper estimate of the Hausdorff measure in~$\B$.
  \begin{lemma} \label{mainlemmagamma}
 Let $M \in {\mathfrak M}_\B$, then
 \begin{equation*}  
   \gamma_\B  (M) \le \mu_\gamma (M) + \omega_\B (M).
\end{equation*}
 \end{lemma}
 \begin{proof}
Let $a > \w (M)$, 
$ \{ A_1, \dots, A_n\}$   be a finite partition of  $\Omega$ and $\{ \varphi_1, \dots, \varphi_m \}$
  a finite set in $\B$  such that, for all $ \ f \in M$,  
there is $j \in \{1, \dots, m\}$ such that 
$
 \mbox{diam}( (f - \varphi_j) (A_i) ) \le a \ \mbox{for} \ i= 1, \dots, n.
$
Set
\[
M_j= \{ f \in M: \ \mbox{diam}((f- \varphi_j)(A_i)) \le a \ \mbox{for} \ i=1, \dots,n \},  \ \  \mbox{  for}  \      \   j=1, \dots,m.
\]
Now,  for each $i=1, \dots,n$,  let $x_i \in A_i$ be fixed. Then
\[
\gamma (\bigcup_{i=1}^n M_j(x_i)  ) = \max_{i=1}^n  \gamma ( M_j(x_i)  )  \le \max_{i=1}^n \gamma(M(x_i)) \le \mu_\gamma (M).
\]
 Next, let $b> \mu_\gamma (M)$ and  let  $\{y^j_1, \dots, y^j_{k_j} \}$  be a $\|\cdot\|$$\mbox{-}b$-net for 
$\bigcup_{i=1}^n M_j(x_i)$ in $Y$. 
Finally, let   $H^j$ be  the set of all functions $h^j:   \{1, \dots n\} \to   \{1, \dots, k_j \}$,   and for  $h^j  \in H^j$,
 define    $\psi_{j,h^j}: \Omega \to Y$ by setting, for $x \in \Omega$,
\[
\psi_{j,h^j} (x)= \sum_{i=1}^n \chi_{A_i} (x) \left( \varphi_j(x) -  \varphi_j (x_i) + y^j_{h^j(i)} \right).
\]
For each fixed $j$,  we show  that $\{\psi_{j,h^j}: h^j \in H^j\}$ is a $\|\cdot\|_\infty$$\mbox{-}(a+b)$-net for $M_j$ in $\B$. 
To this end,  given $f \in M_j$   choose $h^j \in H^j$  
such that  
$
\|f(x_i)- y^j_{h^j(i)} \| \le b,  \  \mbox{for each} \   i \in \{ 1, \dots, n\}.
$
Then for $x \in \Omega$, by letting   $i \in \{ 1, \dots, n\}$ such that $x \in A_i$,
we have
 \begin{eqnarray*}
 \| f(x)- \psi_{j,h^j} (x) \| &=&  \| f(x) - \varphi_j(x) +  \varphi_j(x_i) -  y^j_{h^j(i)} \|
\\
&
 \le &   \| (f - \varphi_j )(x)  -  (f - \varphi_j )(x_i)   \| + \| f(x_i)  -y^j_{h^j(i)} \| \le a + b.
\end{eqnarray*}
Therefore, we obtain
\begin{eqnarray*}
\| f - \psi_{j, h^j}  \|_\infty &= \max_{i=1}^n \sup_{x \in A_i} \| f(x)- \psi_{j, h^j} (x)\|
 \le a+b. \end{eqnarray*}
Hence the  set $\bigcup_{j=1}^m \{ \psi_{j, h^j}: \ h^j \in H^j\} $   is a finite $\|\cdot\|_\infty$$\mbox{-}(a+b)$-net for $M$ in $\B$.     
The arbitrariness of $a$ and $b$ implies
 $ \gamma_B(M) \le \mu_\gamma (M) + \w(M)$, and the proof  is complete. 
\end{proof}

\begin{theorem} \label{maintheoremgamma}
Let  $\X \subseteq \B$ and let $M \in {\mathfrak M}_\X$. Then
 \begin{equation} \label{estimatesgamma}
\max \{ \mu_\gamma(M), \ \frac1{2} \omega_\X (M) \} \le \gamma_\X(M) \le 2 (\mu_\gamma (M) + \omega_\X (M)).  \end{equation}
 \end{theorem}
\begin{proof}
We prove the left inequality.  Let $ a > \gamma_\X (M)$ and let $\{\varphi_1, \dots, \varphi_n\}$ be a $\| \cdot \|_\infty$-$a$-net for $M$ in $\X$. Then it is easy to check that $\{\varphi_1(x), \dots, \varphi_n(x)\}$ 
is a $\| \cdot \|$-$a$-net for $M(x)$ in $Y$. Hence
$\sup_{x \in \Omega} \gamma (M(x)) \le a$, and, by the arbitrariness of $a$, we have 
\begin{equation} \label{left1}
\mu_\gamma(M) \le \gamma_\X(M).
\end{equation}
Next, let $f \in M$ be arbitrarily fixed and $j \in \{1, \dots,n \}$ such that $\|  f - \varphi_j\|_\infty \le a$, then
\begin{eqnarray*}
\mbox{diam}((f - \varphi_j)(\Omega)) &=& \sup_{x,y  \in \Omega} \|(f - \varphi_j)(x) - (f - \varphi_j)(y)  \|  
\\
&\le &  2 \|f - \varphi_j \|_\infty  \le 2 a.
\end{eqnarray*}
By the arbitrariness of $a$, in view of the definition of $\omega_\X (M)$, choosing $\{\Omega\}$ as partition of $\Omega$ and  $\{\varphi_1, \dots, \varphi_n\}$ as finite set in $\X$
we have
\begin{equation} \label{left2}
\omega_\X (M) \le 2 \gamma_\X(M).
\end{equation}
Hence  by (\ref{left1}) and (\ref{left2}), the desired result follows.

The right inequality follows from Lemma~\ref{mainlemmagamma},  taking into account  that
$
 \gamma_\X (M) \le 2 \gamma_\B (M)
$
and   $
 \omega_\B (M) \le  \omega_\X (M)$.  
\end{proof}

  The compactness criteria  given in Corollary~\ref{cor1}  could be deduced as well from  Theorem~\ref{maintheoremgamma}.
Moreover we notice that the set function  $ \mu_\gamma   + \omega_\X$  is a  regular measure of noncompactness in $\X$ equivalent to $\gamma_\X$.

 From Lemma~\ref{mainlemmagamma} and  the left hand side of (\ref{estimatesgamma}) we obtain the following theorem.

\begin{theorem} \label{maintheoremgammaBIS}
Let    $M \in {\mathfrak M}_\B$. Then
 \begin{equation*}  
\max \{ \mu_\gamma(M), \ \frac1{2} \omega_\B (M) \} \le \gamma_\B(M) \le  \mu_\gamma (M) + \omega_\B (M). 
 \end{equation*}
 \end{theorem}
 \begin{corollary} \label{Ambrosettibis}
 Let    $M \in  {\mathfrak M}_\B $  and assume $\omega_\B(M)=0$.
 Then  
 \[
 \gamma_\B (M)= \mu_\gamma(M).
 \] 
 \end{corollary} 
Next, we look at the inequalities in $\T$. The following lemma will be useful.
 \begin{lemma} \label{lemmaAT}     
  If  $\X \subseteq \T$ and $M \in {\mathfrak M}_\X$,
  then
  \[
  \sigma_\gamma (M) \le \gamma_\X (M).
  \]
  \end{lemma}
\begin{proof}
Let $M \in {\mathfrak M}_\X$, $a > \gamma_\X (M)$ and let $\{ \varphi_1, \dots, \varphi_n \}$ be  a $\|\cdot\|_\infty$$\mbox{-}a$-net for $M$ in $\X$.
By hypothesis we have $ \gamma( \bigcup_{i=1}^n \varphi_i(\Omega))=0$, thus 
given  $\eps> 0 $  we find a $\|\cdot\|$-$b$-net $\{\xi_1, \dots, \xi_m\}$   for $ \bigcup_{i=1}^n \varphi_i(\Omega)$ in $Y$.
Then    $\{\xi_1, \dots, \xi_m\}$  is 
a $\|\cdot\|$$\mbox{-}(a+\eps )$-net for $ f(\Omega)$ in $Y$. Indeed, given $f \in M$ and $x \in \Omega$  arbitrarily fixed, choose 
$i \in \{1,\dots,n \}$  such that $\|f(x) -\varphi_i(x)\| \le  a$
and then $j \in \{ 1, \dots,m\}$  such that $\| \varphi_i(x) -\xi_j\| \le  \eps$, so that   
$\|f(x) -\xi_j\| \le  a +\eps$ holds.  Then  $\sigma_\gamma (M) = \gamma (M(\Omega)) \le a+\eps $
 and, by the arbitrariness of $\eps$,  the result follows. 
\end{proof}
  \begin{theorem}  \label{AT} 
If $M \in {\mathfrak M}_\T$, then
  \[
  \max \left\{ \sigma_\gamma (M), \frac1{2} \omega (M) \right\} \le \gamma_\T (M) \le \mu_\gamma (M) + \omega (M).
  \]
  \end{theorem}
  \begin{proof}
  By Lemma~\ref{lemmaAT} we have $\sigma_\gamma (M) \le \gamma_\T (M)$.  Proposition~\ref{propTB}  gives $\omega_\T(M)= \omega(M)$, hence the left inequality follows  from     Theorem~\ref{maintheoremgamma}.
 The right inequality can be proved in the same way  as in  Lemma~\ref{mainlemmagamma}, with the simplifications due to the use of the quantitative characteristic $\omega$ instead  of that of  extended equicontinuity.
  \end{proof}
 We point out  that  Theorem \ref{AT}    recovers the result proved in  \cite[Theorem~2.1 and Proposition~3.1]{AT}.  
Moreover, if    $M \in  {\mathfrak M}_\T $  and   $\omega(M)=0$ we have  $\gamma_\T (M)= \mu_\gamma(M) = \sigma_\gamma(M)$.

     Finally,  the  following example shows that given  a set $M$ in ${\mathfrak M}_\T$, in general   $\gamma_\T (M) \ne \gamma_\B (M)$.    
  \begin{example}
  {\rm
  Let $\Omega = \N$, $Y=\ell_1$ (so $\| \cdot \|$ and $\gamma$ denote  the norm and the Hausdorff measure of noncompactness in $\ell_1$, respectively),  and let $\{e_n\}_{n=1}^\infty$  be the standard basis in $\ell_1$.
  We consider the bounded set  $M= \{f_k: \  k=1,2, \dots   \} $ in $\T (\N, \ell_1)$, where  
     \begin{eqnarray*} 
f_k(n)= \left\{
\begin{array}{lll}
e_k   \quad \mbox{for} \  n=k
\\
[.1in]
 0 \quad  \ \mbox{for} \  n \neq k.
\end{array}
\right.
\end{eqnarray*}
 Then we define  $g  : \N \to \ell_1$  by setting  $g(n)= \frac1{2} e_n$ for $n \in \N$; so to have   $g  \in \B  (\N, \ell_1)$, but $g \notin \T  (\N, \ell_1)$.  
 Thus given $k \in \N$, we have 
    \begin{eqnarray*} 
\|f_k-g \|_\infty &= &\sup_{n \in \N} \| f_k(n)- g(n) \|
=
\max \left\{   \| f_k(k)- g(k) \|,     \  \sup\limits_{\small{\substack{
n \in \N
\\
n \neq k }}}  \| f_k(n)- g(n) \| \right\}
\\
&=&
\max \left\{   \left\| e_k - \frac1{2} e_k \right\|,     \  \sup\limits_{\small{\substack{
n \in \N
\\
n \neq k }}}  \left\|  \frac1{2} e_n \right\| \right\}  = \frac1{2}.
\end{eqnarray*}
This shows that  $\{g\}$ is a $\|\cdot\|_\infty$$\mbox{-}\frac1{2}$-net for $M$ in $\B  (\N, \ell_1)$,   thus $\gamma_{\B(\N, \ell_1)} (M) \le \frac1{2}$.
  Next, on the one hand
\[
 \sigma_\gamma (M)  =  \gamma (M(\N)) =    \gamma \left( \cup_{n \in \N} M(n) \right)
=  \gamma \left( \cup_{n \in \N}   \{  e_n,0 \} \right)=1,
 \]
hence  from Corollary~\ref{AT}  it follows 
  $ \gamma_{TB (\N, \ell_1)} (M)   \ge 1$.
 On the other hand, taking into account that $\gamma_{\B (\N, \ell_1)} (M )  \le  \gamma_{\T (\N, \ell_1)}(M )  \le 2 \gamma_{\B (\N, \ell_1)} (M ) $ we infer 
 $\gamma_{\T (\N, \ell_1)} (M) =1$ and   $\gamma_{\B (\N, \ell_1)} (M) = \frac1{2}$. 
 }
  \end{example}
 \section{ Compactness in the spaces $\ \BC^k(\Omega, Y)$ and $\D^k(\Omega, Y)$} \label{k}
Throughout this section, the differentiability of functions is considered in the Fr\' echet sense.
Let us start introducing the spaces of interest.   As before, $Y$ denotes a Banach space with norm $\|\cdot \|$.  Moreover, we assume  that $\Omega$  is  an open set  in a Banach space $Z$  and  that $k \in \N$ is  fixed. 
 The symbol $\C^k(\Omega, Y)$ will stand for the space  of all $Y$-valued functions defined and  $k$-times continuously  differentiable  on $\Omega$.
 For $f \in \C^k(\Omega,Y)$   we denote by $d^pf: \Omega \to \Linear (Z^p,Y)$,  for $p=0, \dots k$, the differential of $f$ of order~$p$, where
$\Linear (Z^p,Y)$ denotes  the  Banach space of all bounded $p$-linear operators endowed with the standard operator norm
$
 \|   \cdot \|_{on}$.
We have $\Linear (Z^0,Y)=Y$, $d^0f=f$ and  $\C^0(\Omega,Y)=\C (\Omega,Y)$.
 We denote, for each $p \in \{0, \dots,k\}$, by  $\W_p$  the spaces $\Linear (Z^p,Y)$, where $ \W_0=Y$,  thus     each $d^pf$  is an element of  the Banach space $\C (\Omega, \W_p)$,
 endowed with the supremum norm. Moreover, we will denote by $\alpha_p$  and $\gamma_p$ ($p=0, \dots,k$), respectively, the Kuratowski  and the Hausdorff  measure of noncompactness in $\W_p$, where $\alpha_0= \alpha$ and $\gamma_0= \gamma$
$\alpha_0= \alpha$  in $Y$.  
Finally, given   a set $M$ in ${\C}^k (\Omega, Y)$ and  $p\in \{0, \dots,k \}$,  we define \begin{equation} \label{Mp}
 M^p= \{d^p f: \ f \in M \}  \subseteq \C(\Omega, \W_p),
 \end{equation} 
where $M^0=M$.
Then,  for  $x \in \Omega$ and  $A \subseteq \Omega$, the sets $ M^p(x) $ and $ M^p(A)$  will be the subsets of $\W_p$ described according to (\ref{Mdix}) by
\[
 M^p(x) = \{d^p f(x)  : \ f \in M \}, \quad  
 M^p(A) = \{d^p f(A): \ x \in A,   \ f \in M \} . 
 \]
 Now we denote by $\BC^k (\Omega,Y)$ the space consisting of  all functions $f \in \C^k (\Omega,Y)$ which are bounded with all differentials up to the order $k$, made into a Banach space by  the norm
 \[
\|f\|_{\BC^k}= \max\{\|f\|_\infty, \|df\|_\infty, \dots, \| d^k f\|_\infty \},
\]
where  $\|d^pf\|_\infty= \sup_{x \in \Omega} \|d^p f(x) \|_{on}$.
\subsection{ Results in $\ \BC^k(\Omega, Y)$}
For   $M \in {\mathfrak M}_{{\BC}^k}$ we define  the following   quantitative characteristics based on those considered in $\B$, precisely
 \[
 {\mu}_{\alphaBCk} (M)= \max_{p=0}^k \mu_{\alpha_p} (M^p), \quad      {\mu}_{\gammaBCk }(M)= \max_{p=0}^k \mu_{\gamma_p} (M^p)
\] 
and also
\[
 \sigma_{\alphaBCk }(M)= \max_{p=0}^k  \sigma_{\alpha_p} (M^p)  ,
\quad            \sigma_{\gammaBCk }(M)= \max_{p=0}^k  \sigma_{\gamma_p} (M^p).
\] 
A subset $M$ is called {pointwise $k$-relatively compact}   
 if  ${\mu}_{\alphaBCk} (M)=0$ {($ {\mu}_{\gammaBCk} (M)=0$)}.
Further, we introduce
 \begin{eqnarray*} 
 \begin{split}
\omega_{\BC^k} (M)=   \inf \{& \eps >0:  \  \mbox{there  are a finite partition} \ \{ A_1, \dots, A_n\}  \ \mbox{of} \ \Omega
\\
&\mbox{and a finite set} \ \{ \varphi_1, \dots, \varphi_m \} \ \mbox{in} \  \BC^k  \ \mbox{such that, for all}  \ f \in M, 
\\
& \mbox{there is} \ j \in \{1, \dots, m\} \ \mbox{with} \    \max_{p=0}^k \  \mbox{diam}_{\W_p} ( d^p(f - \varphi_j)) (A_i) ) \le \eps \ \mbox{for} \ i= 1, \dots, n\}.
\end{split}
 \end{eqnarray*}
We call  $M$ extendedly $k$-equicontinuous  
 if $\omega_{\BC^k} (M)=0$.
  If $\V$ denotes a   Banach subspace of $\BC^k (\Omega,Y)$ equipped with the same norm, then  $\omega_\V$  is defined as above taking $ \{ \varphi_1, \dots, \varphi_m \} \ \mbox{in} \  \V  $.
Note that for    $k=0$ all these quantities  evidently reduce to  the corresponding quantities introduced in Section \ref{P} and considered in the space $\BC (\Omega,Y)$.
  We are ready to prove the following equivalent relations in the space $\BC^k (\Omega,Y)$. 
The proof  is substantially different from that of    Theorem~\ref{maintheoremalpha}, although it follows the same steps of it.
  \begin{theorem} \label{Theorem in Ck}
Let  $M \in {\mathfrak M}_{{\BC}^k}$, then
 \begin{equation} \label{estimatesCk}
  \max \{ {\mu}_{\alphaBCk} (M), \ \frac1{2} \omega_{{\BC}^k} (M) \} \le 
 \alpha (M) \le  {\mu}_{\alphaBCk} (M) + 2\omega_{{\BC}^k} (M).
\end{equation}
 \end{theorem}
 \begin{proof}
 We prove the left inequality.
 Let $a > \alpha (M)$ and let $\{M_1, \dots, M_n\}$ be a finite cover of $M$ such that $\mbox{diam}(M_i) \le a$ for $i= 1, \dots,n$.
 For all $x \in \Omega$, we have  $\M^p(x) \subseteq \bigcup_{i=1}^n M^p_i(x)$, for $p=0,\dots,k$.
 Then,
\begin{eqnarray*}
 \mbox{diam}_{\W_p} (M^p_i(x)) &=& \sup_{f,g \in M_i} \| d^pf(x)-d^pg(x) \|_{on} \le 
  \sup_{f,g \in M_i} \|d^p  f- d^p g \|_\infty 
   \\
&\le&    \sup_{f,g \in M_i} \| f- g \|_{\BC^k} 
  \\
  &= & \mbox{diam}_{{\BC}^k      }(M_i) \le  a.
 \end{eqnarray*}
  Hence
 $\alpha_p (M^p(x)) \le a$ for all $x \in \Omega$,  which implies $\mu_{\alphaBCk} (M) \le a$.
By the arbitrariness of $b$, we have 
\begin{equation} \label{left3.10}
\mu_{\alphaBCk} (M)  \le \alpha (M).
\end{equation}
 Next, choose $\{\Omega \}$ as a finite partition of $\Omega$ and $\{\varphi_1, \dots, \varphi_n\}$
  as a finite set in ${\BC}^k$ , where  $\varphi_i \in M_i$ for $i=1, \dots, n$.
 Let $f \in M$, fix $i \in \{1, \dots, n \}$ such that $\|f - \varphi_i\|_{{\BC}^k} \le a$.
 Then
 \begin{eqnarray*}
\max_{p=0}^k  \mbox{diam}_{\W_p}   ((d^p f- d^p\varphi_i)(\Omega)) & = &
 \max_{p=0}^k \sup_{x,y \in \Omega} \| d^p(f- \varphi_i)(x)-d^p(f- \varphi_i)(y) \|_{on}
\\
& \le & 2 \max_{p=0}^k  \| f- \varphi_i\|_{{\BC}^k}\le 2a.
 \end{eqnarray*}
Hence $\omega_{{\BC}^k} (M) \le 2a$, and, by the arbitrariness of $a$,
 \begin{equation}  \label{left3.11}
 \omega_{{\BC}^k} (M) \le 2 \alpha (M).
 \end{equation}
 Combining  (\ref{left3.10}) and (\ref{left3.11}), we obtain the left   inequality  of (\ref{estimatesCk}). 

 Now, let us show the right inequality.
  Let $a > \omega_{{\BC}^k} (M)$, 
$ \{ A_1, \dots, A_n\}$    a finite partition of  $\Omega$ and $\{ \varphi_1, \dots, \varphi_m \}$
  a finite set in ${\BC}^k(\Omega, Y)$  such that, for all $ \ f \in M$,  
there is $j \in \{1, \dots, m\}$ such that 
\[
\max_{p=0}^k \mbox{diam}_{\W_p}( (d^p(f - \varphi_j)) (A_i) ) \le a \ \mbox{for} \ \ \  i= 1, \dots, n.
 \]
For each $j=1,\dots,m$ set
 \[
 M_j= \{ f \in M: \   \max_{p=0}^k \mbox{diam}_{\W_p}  ( (d^p(f - \varphi_j)) (A_i) )  \le a \ \mbox{for} \ i=1, \dots,n \}. 
  \]
Fix $x_i \in A_i$ for $i=1, \dots,n$. Then for $p=   0, \dots,k$
\[
\alpha ( \bigcup_{i=1}^n M_j^p(x_i) ) \le \alpha ( \bigcup_{i=1}^n M^p(x_i) ) = \max_{i=1}^n \alpha_p (M^p(x_i)) \le \mu_{\alpha_p}  (M^p)  \le \mu_{\alphaBCk} (M).
\]
 Moreover, let $b >   \mu_{\alphaBCk}(M)$ and,  for each $p= 0 , \dots, k$, let $\{ B^j_1, \dots, B^j_{l_{p(j)}} \}$ be
  a finite cover of  \ $\bigcup_{i=1}^n M_j^p(x_i)$ such that $\mbox{diam} \  B^j_s \le b$ for $s =1, \dots, l_{p(j)}$.
\\
   Let $j \in \{ 1,\dots,m \}$ be arbitrarily fixed, and   let  $H^j$ be  the set of all functions $h^j=(h^j_0,\dots ,h^j_k)$    where,  
for  each $p\in \{0, \dots k \}$, $h^j_p$ maps      $\{1, \dots n\}   \  \mbox{into}  \  \ \{1, \dots, l_{p(j) }\}  $.
Set 
 \[
 M_{j,h^j}=\{ f \in M_j : \mbox{for} \ p=0, \dots,k, \ d^p f(x_i) \in B^j_{h^j_p (i)}, \ \mbox{for} \ i=1,\dots, n \}.
 \]
 Then $\{ M_{j,h^j}: \ h^j \in H^j  \}$ is a finite cover of $M_j$ in $\BC^k$. Moreover, for all $h^j \in H^j$, we have 
 \begin{eqnarray*}
 \mbox{diam}_{\BC^k}  (M_{j,h^j})  & = &\sup_{f,g \in M_{j, h^j}}  \| f-g \|_{{\BC}^k} = \sup_{f,g \in M_{j,h^j} }\max_{p=0}^k \| d^p(f-g) \|_\infty
 \\
 &= &
\max_{p=0}^k   \sup_{f,g \in M_{j,h^j} } \|d^p( f-g) \|_\infty .
 \end{eqnarray*}
Now, for  all $ p\in \{ 0, \dots k\}$ and  for all  $f,g \in M_{j,h^j} $, we have
 \begin{eqnarray*}
 \| d^pf-d^pg \|_\infty &= &\max_{i=1}^n \sup_{x \in A_i}  \|d^pf(x)-d^pg(x)\|_{on}
 \\
 &\le & \max_{i=1}^n \sup_{x \in A_i}   (  \| d^p(f- \varphi_j)(x)- d^p(f- \varphi_j)(x_i) \| 
 \\ 
 && \ \ + 
  \| d^p(g- \varphi_j)(x)- d^p(g- \varphi_j)(x_i) \| 
+ \|d^pf(x_i)-d^pg(x_i ) \|  )  
  \\
  & \le & 2   \max_{p=0}^k \mbox{diam}_{\W_p}( (d^p(f - \varphi_j)) (A_i) )  +  \max_{i=1}^n \sup_{x \in A_i}  \|d^pf(x_i)-d^pg(x_i ) \|
   \le   b +2a.
    \end{eqnarray*}
  Therefore, $ \mbox{diam}_{\BC^k}  (M_{j,h^j})  \le b+2a$ so that $ \alpha (M_j) \le b+2a$.
  Since $\{M_j: j=1, \dots,m \}$   is a finite cover of $M$, 
 we get
$ \alpha (M) \le b+2a$ 
  and by the arbitrariness of $a$ and $b$ we find  $\alpha (M) \le \mu_{\alphaBCk } (M) + 2 \omega_{{\BC}^k} (M)$.
 \end{proof}

We obtain the following  new criterion of compactness in   $\BC^k(\Omega, Y)$.
\begin{corollary} \label{compactness in BCk}
A  subset $M$ of $\BC^k (\Omega,Y)$ is relatively compact if and only if it is bounded, extendedly $k$-equicontinuous and pointwise $k$-relatively compact.
 \end{corollary}
 In the literature  (see, for example,  \cite{iraniano, BCW,BMR, cianciaruso})  there are results that  characterize compactness in  $\BC^k(\Omega,Y)$, or proper subspaces of it,
for particular  $\Omega$ or $Y$. 
 Let us now  consider, for example,  the case $Z= \R$. Then each space   $W_p=\Linear (Z^p,Y)$, for $p= 0, \dots,k$,   can be identified with $Y$ itself. 
  Therefore,  whenever  $\Omega$ is an open subset of $\R$ and $M$ a bounded subset of $\BC^k (\overline{\Omega}, Y)$, 
   the estimates (\ref{estimatesCk})  hold with   $  {\mu}_{\alphaBCk} (M)= \max_{p=0}^k \mu_{\alpha} (M^p)$.  We also notice that in the definition of   $\omega_{\BC^k}(M)$ , in such a case,  all diameters will be actually calculated in $Y$. 
  Hence we obtain the following new criterion of compactness in $\BC^k (\overline{\Omega}, Y)$,  which in particular  recovers the case  $\BC^k ([0, + \infty), Y)$(cf.~\cite{cianciaruso}).
 \begin{corollary}~\label{cianc}
 Let $\Omega$ be an open  subset of $\R$ and let $M$ be a subset of  $\BC^k  (\overline{\Omega}, Y)$.
Then $M$ is relatively compact if and only if   each $M^p$   for $p \in \{0, \dots,k \}$  is bounded and pointwise  relatively compact and $M$ is extendedly $k$-equicontinuous. 
 \end{corollary}
Going back to the general case, from  (\ref{estimatesCk}) we obtain an Ambrosetti-type formula also in the space $\BC^k(\Omega, Y)$,
and a regular measure of noncompactness equivalent to that of  Kuratowski.
  \begin{corollary} 
Let  $M \in {\mathfrak M}_{{\BC}^k}$ be  extendedly $k$-equicontinuous.
 Then
\begin{equation*}  
 \alpha (M) = {\mu}_{\alphaBCk} (M).
\end{equation*}
 \end{corollary}
  \begin{corollary}  \label{newMNCbis}   The set function   $ {\mu}_{\alphaBCk} + 2\omega_{{\BC}^k}:  {\mathfrak M}_{{\BC}^k} \to[0, +\infty)$ is a  regular measure of noncompactness in $\BC^k(\Omega, Y)$ equivalent to the Kuratowski measure $\alpha$.
\end{corollary}
 \begin{remark} \label{rem}
 {\rm
 If $\V$ is a Banach  subspace of $\BC^k(\Omega, Y)$, Theorem \ref{Theorem in Ck} and subsequent Corollaries hold true in $\V$.  
 }
 \end{remark}       
 Now we focus  our attention on  the Banach space $\TBC^k(\Omega, Y)$  consisting of all   functions $f \in \BC^k(\Omega, Y)$ which are compact with all differentials  up to the order $k$. The following remark shows that  the hypothesis that each $d^pf$ ($p= 1, \dots, k$) is compact is not redundant.
\begin{remark}
{\rm
It is well known (see, for instance, \cite{DS}) that if $f \in \BC^k(\Omega, Y)$ is a compact function, then  for each $x \in \Omega$ the differentials $d^pf(x)$ of $f$ at $x$, for $p \in \{0, \dots k \}$,  are  compact  linear operators.
On the other hand, there are compact functions  $f \in \BC^k(\Omega, Y)$  such that  the functions $d^pf: \Omega \to \Linear (Z^p,Y)$ are not compact.  For example, let us consider 
 $Z =\R$, $ \Omega = \bigcup_{n=1}^\infty I_n$ with $I_n=\left( n - \frac1{2n}, n + \frac1{2n}\right)$, $Y= \ell_1$ and  $\{e_n\}_{n=1}^\infty$ the standard basis in $\ell_1$.
Then we define   $f \in \BC^1(\Omega, \ell_1)$   by setting
\[  
f(x)=
 (x-n) e_n \quad  \mbox{for} \ x \in I_n  \ \   ( n=1,2 \dots).
\]
Clearly $f$ is a compact function.  On the other hand, since
$
df(x)=   e_n $ if $ x \in I_n$, we deduce $df (\Omega)=\{e_1, \dots, e_n, \dots \}$ and this shows that  $df$
is not compact.
}
\end{remark} 
Now given $M \in {\mathfrak M}_{{\TBC}^k}$,  we  define $\overline{\omega}(M) $  extending  the definition  of  $\omega$ given in (\ref{omega}). We set
 \begin{eqnarray*} 
 \begin{split}
\overline{\omega} (M)=   \inf \{ \eps >0:  \ & \mbox{there  are a finite partition} \ \{ A_1, \dots, A_n\}  \ \mbox{of} \ \Omega
\\
&\ \mbox{such that, for all}  \ f \in M,  \ 
 \max_{p=0}^k \  \mbox{diam}_{\W_p} ( d^p(f (A_i) ) \le \eps \ \mbox{for} \ i= 1, \dots, n  \},
 \end{split}
 \end{eqnarray*} 
By the definition, it is immediate to see that
 \begin{equation} \label{omega segnato}
 \overline{\omega} (M) = \max_{p=0}^k \omega       (M^p).
 \end{equation}
Next, repeating the  arguments of  Proposition \ref{propTB}, given  $M \in {\mathfrak M}_{{\TBC}^k}$ we find
$
  \omega_{{\TBC}^k} (M) =  \overline{\omega} (M).
$
Arguing similarly as  in Lemma \ref{lemma1},  given $M \in {\mathfrak M}_{{\TBC}^k}$   we can prove $ \sigma_{\overline{\alpha}} (M^p) \le \alpha (M)$.
 Therefore, we obtain the following result as consequence of Theorem~\ref{Theorem in Ck}.
 \begin{theorem}
 Let  $M \in {\mathfrak M}_{{\TBC}^k}$.
 Then 
   \begin{equation*}.  
  \max \{ \sigma_{\alphaBCk} (M), \ \frac1{2} \overline{\omega}(M) \} \le 
 \alpha_{{\TBC}^k} (M) \le  {\mu}_{\alphaBCk} (M) + 2 \overline{\omega}(M).
\end{equation*}
 \end{theorem} 
Observe that, if $\Omega$ is an open bounded   subset of $\R^n$, then $\TBC^k (\overline{\Omega}, Y)= \BC^k   (\overline{\Omega}, Y)=    \C^k  (\overline{\Omega}, Y)$.
  \begin{corollary}
If $\Omega$ be an open bounded subset of $\R^n$ and let $M \in    {\mathfrak M}_{\C^k  (\overline{\Omega}, Y)}$.
 Then 
   \begin{equation*}      
  \max \{ {\sigma}_{\alphaBCk} (M), \ \frac1{2} \overline{\omega}(M) \} \le 
 \alpha_{{\C}^k} (M) \le  {\mu}_{\alphaBCk}  (M) + 2 \overline{\omega}(M).
\end{equation*}
 \end{corollary}
{  Next if  $\Omega$  is an open subset of $R^n$ ed   $Y= \R$,  among others, some results of \cite{iraniano} are recovered.}
{ To this end, let us mention that if $ f \in  \BC^k  ({\Omega},  \R)$}
then
\[
\|f\|_{ \BC^k}=  \max_{0 \le | \alpha | \le k} \|D^\alpha f \|_\infty,
\]
 where $\|D^\alpha f \|_\infty = \sup \{ |D^\alpha (x)| : x \in \Omega \} $, $|\alpha|= \alpha_1 + \dots + \alpha_n$ and  
$
D^\alpha f= \frac{\delta^{\alpha_1}}{\delta x_1^{\alpha_1}}  \dots \frac{\delta^{\alpha_n}}{\delta x_n^{\alpha_n}} f.
$
Keeping in mind this and taking into account (\ref{omega segnato}) we deduce the following compactness criteria.
  \begin{corollary}
Let $\Omega$ be an open  bounded subset of $\R^n$ and let   $M $   be a subset of   $ \C^k  (\overline{\Omega},  \R)$.
 Then the following are equivalent:
 \begin{itemize}
 	\item[(i)] 
 $M$ is $\| \cdot \|_{\C^k}$-relatively compact,
 	 \item[(ii)]
 $M$ is $\| \cdot \|_{\C^k}$-bounded and  $\overline{\omega}(M)=0$,
    	\item[(iii)] 
$M^p$ is $\| \cdot \|_\infty $-bounded and equicontinuous for all $p=0, \dots, k$,
       	\item[(iv)] 
$M^\alpha= \{ D^\alpha f : f \in M \} $ are $\| \cdot \|_\infty $-bounded and equicontinuous for all $0 \le | \alpha | \le k$.
 \end{itemize}
 \end{corollary}
 In particular,   the above condition (iv)    recovers Theorem 2.1 of \cite{iraniano}.
Moreover, 
 denoting by $| \cdot |_n$   the Euclidean norm in $\R^n$ we define the subspace $\C^k_0 (\R^n, \R)$ of $\BC^k (\R^n, \R)$  as follows:
\[
\C^k_0 (\R^n, \R) =\{ f \in  \C^k (\R^n, \R) : D^\alpha f \in \C_0 \ \mbox{for} \  0 \le | \alpha | \le k \},
\]
where  ${\C_0= \{ f \in \BC^k(\R^n, \R) : \lim_{ |x|_n \to \infty} f(x)=0 \}}$, with the norm
$ \|f\|_{\C^k_0} = \max_{0 \le | \alpha | \le k} \|D^\alpha f \|_\infty$.
Then,  from Corollary    \ref{compactness in BCk}  and in view of Remark \ref{rem}, we have a compactness criterion in the space 
$\C^k_0 (\R^n, \R) $ (cf. \cite[Theorem 3.1]{iraniano}).

We complete this section by stating, without proofs,  estimates and precise formulas for the Hausdorff measure of noncompactness,  involving the quantitative characteristics  $  {\mu}_{\gammaBCk }$ and  $  {\sigma}_{\gammaBCk }$,  in the spaces $\BC^k   (\Omega, Y)$  and  $\TBC^k (\Omega, Y)$.

\begin{theorem} 
Let  $M \in {\mathfrak M}_{{\BC}^k}$, then
 \begin{equation*}
\max \left\{   {\mu}_{\gammaBCk }(M) , \ \frac1{2}   \omega_{{\BC}^k}  (M) \right\} \le \gamma_{{\BC}^k}  (M)  \le  2(  {\mu}_{\gammaBCk }(M)  +  \omega_{{\BC}^k} (M)).  \end{equation*}
\end{theorem}
\begin{theorem} 
Let  $M \in {\mathfrak M}_{{\TBC}^k}$, then
 \begin{equation*}
\max  \left\{   {\sigma}_{\gammaBCk }(M) , \ \frac1{2}   \overline{\omega}  (M) \right\} \le \gamma_{{\TBC}^k}  (M)  \le      {\mu}_{\gammaBCk }(M)  +  \overline{\omega}(M).  \end{equation*}
\end{theorem}
Therefore, if $M \in {\mathfrak M}_{{\TBC}^k}$   and  $   \overline{\omega}(M)=0$,  we have the formulas   $ \gamma_{{\TBC}^k}  (M) = 
 {\sigma}_{\gammaBCk }(M)$ or    $ \gamma_{{\BC}^k}  (M)  = {\mu}_{\gammaBCk }(M)$ .

\subsection{Results in $\D^k( \Omega, Y)$  }
{Finally, we apply  results of Section \ref{B}  to derive compactness results in   $\C^k(\Omega, Y)$ made into  a complete locally convex space  by the topology  $\tau$ of compact convergence for all differentials, 
 i.e. the topology generated by 
 the family of seminorms
 \[
\|  f \|_{\C^k, \U} = \max \{  \sup_{x \in \U} \|f(x)\|, \sup_{x \in \U} \|d f(x)\|, \dots, \sup_{x \in \U} \|d^kf(x)\|\}
  \quad  \U \in \K,
 \]
where the symbol 
$\K$  denotes the family of all compact subsets of $\Omega$. 
We  set   $\D^k( \Omega, Y)= ( \C^k(\Omega, Y), \tau) $.
In particular,  $  \D^0 (\Omega, Y)$ reduces to the space, simply denoted by  $\D (\Omega, Y)$, of all continuous functions from $\Omega$ to $Y$  endowed with the usual topology of uniform convergence on compacta}.
Further,  for a fixed  $\U$ in $ \K$,  we denote by $\D^k_\U( \Omega, Y)$ the complete seminormed space   of all $k$-times continuously  differentiable functions  endowed  with the seminorm $ \| \cdot \|_{\C^k,\U}$.

We use the  notation $ {\mathfrak M}_{\D^k}$   for the family of all $\tau$-bounded subsets of $\D^k( \Omega, Y)$.
 Let us now equip the linear space  of all functions from  $\K$ to $[0, + \infty)$  with the usual order and with the topology of pointwise convergence.
 Then, according to \cite[Definition~1.2.1]{SA},   for a subset $M$ of $ {\mathfrak M}_{\D^k}$,   the Kuratowski  and the Hausdorff   
  measures of noncompactness generated by the family of seminorms 
$  \{ \|  \cdot \|_{\C^k,\U} \}_{ \U \in \K}$ are functions 
$
  \alpha_{\D^k } (M), \ \ \gamma_{\D^k  } (M) :  \ \K \to [0, +\infty)
$
where
$
 \alpha_{\D^k } (M)  (\U)  =  \alpha_{\D^k_\U } (M)  
$,  that is, $
 \alpha_{\D^k } (M)  (\U) $ is the Kuratowski measure of noncompactness of $M$  with respect to the seminorm $\|  \cdot \|_{\C^k, \U}$, and analogously
$
 \gamma_{\D^k } (M) (\U)= \gamma_{\D^k_\U } (M)  
 $. 

We refer to  \cite[Theorem 1.2.3]{SA} for the properties of  these generalized measures of noncompactness.
In a similar way,  we will introduce the quantitative characteristics useful to prove our estimates as functions from  $\K$ to $[0, + \infty)$ .
To this end, for  $M \in  {\mathfrak M}_{\D^k}$,   and $p=0, \dots, k$,   we define $M^p$ as in  (\ref{Mp}), and, given $x \in \Omega$ and $\U \in \K$,  we define consequently also $M^p(x)$ and $M^p(\U)$. 
 Moreover  for $M \in  {\mathfrak M}_{\D^k}$ and $\U \in \K$ we use the following notations
\begin{eqnarray*} 
 \mu_{{\overline{\alpha}}, \U}  (M) &=\max_{p=0}^k  \mu_{\alpha_p , \U} (M^p)   \qquad \mbox{with}  \qquad       \mu_{\alpha_p , \U} (M^p)= \sup_{x \in \U }  \alpha_p (M^p (x)),
\\
\mu_{{\overline{\gamma}}, \U}  (M) &=\max_{p=0}^k  \mu_{\gamma_p , \U} (M^p)               \qquad \mbox{with}  \qquad             \mu_{\gamma_p , \U} (M^p)= \sup_{x \in \U }  \gamma_p (M^p (x)),
\\
\sigma_{{\overline{\alpha}}, \U}  (M) &=\max_{p=0}^k  \sigma_{\alpha_p , \U} (M^p)               \qquad \mbox{with}  \qquad               \sigma_{\alpha_p , \U} (M^p)= \sup_{x \in \U }  \alpha_p (M^p (x)),
\\
\sigma_{{\overline{\gamma}}, \U}  (M) &=\max_{p=0}^k  \sigma_{\gamma_p , \U} (M^p)            \qquad \mbox{with}  \qquad        \sigma_{\gamma_p , \U} (M^p)= \sup_{x \in \U }  \gamma_p (M^p (x)),
 \end{eqnarray*}
 and
 \begin{eqnarray*}
 \begin{split}
  {\overline{\omega}}_\U (M )  = \max_{p=0}^k  \Big\{     \inf \{ &\eps >0:    \ \mbox{there is a finite partition} \ \{ A_1, \dots, A_n\}  \ \mbox{of} \ \U
\\
&   \mbox{such that, for all}  \ f \in M, \  \mbox{diam}_{\W_p} d^pf(A_i) \le \eps , \  \ i= 1, \dots, n\}    \Big\}. 
\end{split}
\end{eqnarray*}
Now,  for a given   $M \in  {\mathfrak M}_{\D^k}$ we define  the set functions
\[
 \mu_{\alpha_\tau} (M),  \ \mu_{\gamma_\tau}    (M),  
 \ \sigma_{\alpha_\tau} (M) ,    \ \sigma_{\gamma_\tau} (M),   \ \omega_{\D^k} (M): \ \K \to [0, + \infty)
\]
     by setting for $\U \in \K$
\begin{eqnarray*} 
  \mu_{\alpha_\tau}(M)   (\U)    &=  \mu_{{\overline{\alpha}}, \U}  (M)  
\\
 \mu_{\gamma_\tau} (M)   (\U)    &=  \mu_{{\overline{\gamma}}, \U} (M)   
\\
  \sigma_{\alpha_\tau} (M)   (\U)    &=  \sigma_{{\overline{\alpha}}, \U}  (M)\\
  \sigma_{\gamma_\tau} (M)   (\U)    &=  \sigma_{{\overline{\gamma}}, \U}  (M)
   \end{eqnarray*}
and 
\[
\omega_{\D^k}  (M )  (\U)= {\overline{\omega}}_\U (M ).
 \]
A set   $M \subset {\mathfrak M}_{{\D}^k}$ will be called   { pointwise $\tau$-relatively compact} 
if  ${\mu}_{\alpha_\tau} (M)=0$ or $ {\mu}_{\gamma_\tau} (M)=0$,
 and   {  $\tau$-equicontinuous}   
 if $\omega_{\D^k} (M)=0$.
  \begin{theorem} \label{Theorem in Dk}
Let  $M \in {\mathfrak M}_{{\D}^k}$, then
 \begin{equation*} 
  \max \{ {\mu}_{\alpha_\tau} (M), \ \frac1{2}    \omega_{\D^k}  (M) \} \le 
 \alpha_{{\D}^k} (M) \le  {\mu}_{\alpha_\tau} (M) + 2 \omega_{\D^k}  (M).
\end{equation*} 
 \end{theorem}
 \begin{proof}
 Let  $M \in {\mathfrak M}_{{\D}^k}$. We have to prove, for each $\U \in \K$
  \begin{equation*}
  \max \{ {\mu}_{\overline{\alpha}, \U} (M), \ \frac1{2} \overline{\omega}_\U (M) \} \le 
 \alpha_{{\D}^k_\U} (M) \le  {\mu}_{\overline{\alpha}, \U} (M) + 2         \overline{\omega}_\U (M).
\end{equation*}
Let us  make an intermediate step. Let   $(\W, \| \cdot \|_\W)$ a given Banach space.
Coherently with our previous notations, 
we denote by  $\D_\U  (\Omega, \W)$  the complete  seminormed space  $(\C  (\Omega, \W), \|\cdot \|_{\C,\U} )$, where 
\[
\|f \|_{\C,\U}  = \sup_{x \in \U} \|f(x)\|_\W, 
\]
and we focus our attention on this space.
Set $N=\{f \in  \D_\U  (\Omega, \W): \|f\|_{\C,\U} =0 \}$
 and let us still denote by $\D_\U  (\Omega, \W)$ the Banach quotient space $\D_\U  (\Omega, \W)/ N$ of equivalence classes,  by $f$  the  class $f+ N$ of $\D_\U  (\Omega, \W)/ N$ and the same for the norm  $\|\cdot \|_{\C,\U}$.
Then, let us observe that the Banach space $\D_\U  (\Omega, \W)  $ is isometric to the Banach space $(\C(\U, \W), \| \cdot \|_\infty)$.
Therefore, for $p=0, \dots, k$,  the quotient Banach spaces $\D_\U  (\Omega, \W_p)  $, endowed with the norm $ \|\cdot\|_{\C,\U}$, are isometric to the Banach spaces $(\C(\U, \W_p), \| \cdot \|_\infty)$, so that    $ \alpha_{\C  (\U, \W_p)}     (M^p) = \alpha_{\D_\U  (\Omega, \W_p)}     (M^p)$.
Hence in view of Corollary~\ref{coralpha},  for each $p$, we have   
\[
  \max \left\{          \sigma_{{\alpha_p}, \U} (M^p), \frac1{2} \omega_\U (M^p) \right\} 
  \le  \alpha_{\D_\U  (\Omega, \W_p)}       (M^p)
   \le  \mu_{{\alpha_p}, \U} (M^p) + 2 \omega_\U (M^p).
  \]     
Taking the maximum for $p=0, \dots, k$ we have
\[
  \max \left\{ \sigma_{{\alpha}, \U} (M), \frac1{2} \omega_\U (M) \right\} 
  \le  \alpha_{ \D^k_\U}     (M)
   \le \mu_{\alpha, \U} (M) +  2 \omega_\U (M),
  \]
  as desired.
\end{proof}

\begin{corollary} \label{compactness in Dk}
A  subset $M$ of $\D^k (\Omega,Y)$ is relatively compact if and only if it is bounded,   
\\
$\tau$-equicontinuous and pointwise $\tau$-relatively compact.
 \end{corollary}
 Let us observe that in the spaces $\D^k(\Omega, Y)$  the function
$ \mu_{\alpha_\tau}   + 2 \omega_{\D^k}$  is a  regular generalized measure of noncompactness 
 equivalent to the Kuratowski measure $\alpha_{\D^k}$.

 We also underline that, for  $\tau$-equicontinuous sets $M$ of $\D^k (\Omega,Y)$, we obtain the formula  $\alpha_{\D^k}(M)=\mu_{\alpha_\tau}(M)$.
 Moreover, as a  particular case of Theorem~\ref{Theorem in Dk},  we obtain estimates for the  Kuratowski  measure in the space  $\D^k (\R^n, \R)$ (defined for instance in \cite{KN})  and the consequent compactness criterion.
For $k=0$, Corollary~\ref{compactness in Dk}  is a special case of the well known general Ascoli-Arzel\` a theorem (\cite[Theorem~18]{K}).
Finally, the same reasoning of Theorem \ref{Theorem in Dk},  using Theorem \ref{AT}, leads to the following inequalities,
which estimate the Hausdorff measure of noncompactness  $\gamma_{\D^k} (M) $ of sets $M \in {\mathfrak M}_{{\D}^k}$,
\[
  \max \left\{ \sigma_{\gamma_\tau} (M), \frac1{2} \omega_{\D^k} (M) \right\} \le \gamma_{\D^k} (M) \le \mu_{\gamma_\tau} (M) + \omega_{\D^k} (M),
  \]
which,  when $\omega_{\D^k}(M)=0$,  gives the formula   $ \gamma_{\D^k} (M) = \sigma_{\gamma_\tau} (M)$ or =   $ \gamma_{\D^k} (M) =\mu_{\gamma_\tau} (M)$.

 \section{A remark in $\B (\Omega, Y)$  when $Y$ is a  Lindenstrauss space}
A real Banach space $Y$  is said to be an $L_1$-predual provided its dual $Y^*$ is isometric to $L^1(\mu)$ for some measure $\mu$.
Such  spaces are often referred to as  Lindenstrauss spaces and   play a central role in the Banach space theory.
The Banach space $C(K)$ of real-valued functions defined and continuous on the compact Hausdorff space $K$, under the supremum norm, is the most  natural example of a Lindenstrauss space.
 We mainly refer to \cite{R1, Casini1, LW, R2}, and to \cite{Lacey} for a survey of results on such spaces. 
Let us recall that given   a bounded subset $H$   of $Y$,
the  Chebyshev radius $r_C(H)$ is defined as the infimum of all numbers  $c>0$  such that $H$ can be covered with a ball of a radius $c$. Thus we have
$
r_C(H)=\inf \{ c>0: \ y \in Y, \ H \subseteq B(y,c) \}  
$
 with $  \frac 1 2 {\rm diam} (H) \le r(H) \le  {\rm diam} (H)$.
A point $\bar{z} \in Y$ is said to be a Chebyshev centre of~$H$
if  $H \subseteq B(\bar{z}, r_C(H))$. The set $H$ is said to be { centrable} 
 if
$r_C(H)= \frac 1 2 {\rm diam} (H)$.
 In  \cite[Theorem 1]{R2}  Lindenstrauss spaces are characterized    as  those Banach spaces   in which every finite set is centrable. Moreover, if $Y$ is  a Lindenstrauss space then every finite set has a Chebyshev centre  
and      every compact  set   is centrable (cf.    \cite[Corollary 1 and Remark 1]{R2}) .
 Whenever $Y$ is a Lindenstrauss space  we find a better lower estimate  for  the Kuratowski measure of noncompactness of bounded and pointwise relatively compact subsets of  the space $\B$. 
 \begin{proposition} Assume that   $Y$ is a Lindenstrauss space and   that  $M  \in  {\mathfrak M}_\B$ is  pointwise relatively compact.
 Then
 $
\omega_\B (M) \le \alpha(M).
$
\end{proposition}
\begin{proof}
 Let $a > \alpha (M)$ and choose  $M_1, \dots, M_n$  such that 
 $M=\cup_{i=1}^n M_i$ and ${\rm diam} M_i \le a$ for $i=1, \dots n$.
  Let $\delta >0$ be arbitrarily fixed. Let $i \in \{1, \dots,n \}$.   Fix  $x \in \Omega$ and let  $F_{i,x} \subseteq Y$ be a finite   inner $\| \cdot \|$-$\delta$-net  for $M_i(x)$. 
Let $z_{i,x}$ be a Chebyshev center of $F_{i,x}$  in $Y$, so that
$
  F_{i,x} \subseteq B (z_{i,x},  \ r_C(F_{i,x}))$, where, $ r_C(F_{i,x})= \frac 1 2 \mbox{diam} (F_{i,x})$.
By the hypothesis $\overline {M_i(x)}$  is a compact set, so that it is centrable, that is, 
$
r_C (\overline {M_i(x)}) = \frac 1 2 \mbox{diam} (\overline {M_i(x)}).
$
Now, we define the mapping $ \varphi_{i}: \Omega \to Y$ by setting  $
\varphi_{i } (x)= z_{i,x}$, for $x \in \Omega$.
Then,  for $f \in M_i$  arbitrarily fixed, we have
\begin{equation} \label{eq1}
\|f- \varphi_i\|_\infty 
   \le  \frac 1 2 \mbox{diam} (M_i) + \delta.
\end{equation}
Indeed, 
for each $ x \in \Omega$ choose $y_{i, x} \in  F_{i,x}$ such that
$
\|f(x)-  y_{i, x}  \| \le \delta.
$ 
Hence we have
\begin{eqnarray*}
\|f- \varphi_i\|_\infty &=& \sup_{  x \in \Omega} \|f(x)- \varphi_i(x)\|  =\sup_{ x  \in \Omega} \|f(x)- z_{i,x}\|
\\[.1in]
&
  \le &
   \sup_{x \in \Omega}  \left(  \|f(x)- y_{i,x}\| +  \|z_{i,x} - y_{i,x}\| \right)    \le    \sup_{x \in \Omega}   \|z_{i,x} - y_{i,x}\|    + \delta
    \\[.1in]
&
\le &
  \frac 1 2  \sup_{x  \in \Omega}  \mbox{diam} (F_{i,x} )+ \delta  \le    \frac 1 2  \sup_{x  \in \Omega}  \mbox{diam} (M_i(x)) + \delta 
    \\[.1in]
&
   \le  & \frac 1 2 \mbox{diam} (M_i) + \delta .
  \end{eqnarray*}
 We also get $\varphi_i \in \B$.
Finally, let  $f \in M$ and choose $i$ such that $f \in M_i$.
Then using (\ref{eq1}) we find 
\begin{eqnarray*}
\mbox{diam} \left( (f- \varphi_i)(\Omega) \right) & =  &\sup_{x,y \in \Omega}  \|    (f- \varphi_i)(x)-  (f- \varphi_i)(y) \|        
\\[.1in]
&
  \le &
 \sup_{x,y \in \Omega} \left(   \|    (f- \varphi_i)(x) \| + \| (f- \varphi_i)(y) \|     \right) \le      2  \| f- \varphi_i \|_\infty
 \\[.1in]
&
  \le &
   2 \left( \frac 1 2 \mbox{diam} (M_i) + \delta \right) =    \mbox{diam} (M_i) + 2 \delta  \le a + 2 \delta.
\end{eqnarray*}
Taking $\{ \Omega \} $ as a partition of $\Omega$ and $\{ \varphi_1, \dots, \varphi_n \} $ as a finite set in $\B$,  from 
 the arbitrariness of $a$ and $\delta$, it follows
$
\omega_\B(M) \le \alpha (M),
$
which is the thesis.
 \end{proof}
 
 Combining the previous result with Theorem \ref{maintheoremalpha} we derive the following estimates.
 \begin{theorem} \label{Lin} Assume  that  $Y$ is a Lindestrauss space and that  $M  \in  {\mathfrak M}_\B$  is pointwise relatively compact.
 Then
\begin{equation*}
\omega_\B (M) \le \alpha(M) \le 2 \omega_\B (M).
\end{equation*}
\end{theorem}
The following two examples show that the inequalities given in Theorem~\ref{Lin} are the best possible.
\begin{example} \label{Lin 1}
{\rm
  Let   $Y $ be an infinite-dimensional Lindenstrauss space  with origin $\theta$, and let $\B=\B(Y,Y)$.
Let $(y_n)_n$ be a sequence in $Y$ such that
$\|y_n -y_m \| \ge 2$ when   $n \ne m$.
We now consider the closed balls  $B(Y)$  and $B(y_n,1)= y_n+ B(Y) $ for all $n=1, 2, \dots$,  which for short we will denote by $B$ and $B_n$, respectively.  Clearly, the sets of the  sequence $(B_n )_n$ are pairwise disjoint.
Let us define $f_n:Y \to Y$ for $n=1, 2, \dots $, by setting
\[
f_n(y)= \begin{cases}
  \theta  \quad &\mbox{for}  \ \  y \notin B_n
\\ 
y-y_n \quad&\mbox{for}  \ \  y \in B_n.
\end{cases}
\]
Set $M = \{ f_n: \ n=1,2,\dots \}$. 
Given $y \in Y$, $M(y)={\theta}$ if $y \notin \cup_{k=1}^\infty B_k$ or $M(y) = \{\theta, y-y_k \}$ if $y \in B_k$, so that $M$ is pointwise relatively compact.
Now, let us observe
that $f_n (Y)= f_n(B_n)= \{ y-y_n: \ y \in B_n\}= B$, for all $n$, hence  by (\ref{DUE})  
 we get
 \[
\omega ( \{f_n \} ) = \alpha (f_n(Y))= \alpha (B)= 2,
\]
so that $\omega (M) \ge 2$.
On the other hand, since  $\mbox{diam} (f_n(Y))= \mbox{diam} (B)=2$ for all $n$, taking $\{Y \} $ as partition of $Y$  we obtain $\omega (M) \le 2$, thus $\omega (M)=2$.
\\
{ Next,  we prove $\alpha (M)=1$. To this end, having in mind the definition of $f_n$,  first we notice 
   that  for all $n, m \in \N$ with $n \ne m$ we have
\[
 \|f_n - f_m \|_\infty
=    \sup_{y \in Y}\|f_n(y) - f_m (y)\|=1  
\]
so that, considering the  the Istratescu measure of noncompactness of $M$, we have $\beta (M) \ge 1$.
{  Since $
  \beta(M) \le \alpha (M)
$, we have
$
  \alpha (M) \ge 1 .
$}
On the other hand
\[
 \mbox{diam} (M)= \sup_{n,m \in \N} \|f_n - f_m \|_\infty
 =1,
\]
so that $\alpha (M) \le 1$, and our assert follows.}
\\
Now we show $\omega_\B (M)=1$.
Set $\varphi (y)= \frac 1 2 \sum_{n=1}^\infty f_n(y)$, for all $y \in Y$.  Then to evaluate $\omega_\B (M)$ we consider $\{ Y \}$   as a partition of $Y$  and
$ \{ \varphi \}$ as 
a finite set in $\B$.
 Then given $y,z \in Y$ and $n \in \N$, we have
\[
\| f_n(y)- \varphi (y)  - f_n(z)+ \varphi (z) \| = \frac 1 2 \| f_n(y)- f_n (z) \|= \frac 1 2 \| y-z \| \le 1 \ \  \mbox{if} \ y,z \in B_n
\]
and
$
\| f_n(y)- \varphi (y)  - f_n(z)+ \varphi (z) \|  \le \frac 1 2  \ \  \mbox{if} \ y,z  \ \mbox{are not simultaneously  both in} \  B_n.
$
Consequently { $\mbox{diam} (f_n- \varphi)(Y)) \le 1$,  } which implies $\omega_\B (M) \le 1$.
Assume by contradiction that  $\omega_\B (M) < 1$  and let $\omega_\B (M) = 1
- \delta$. Then there are   a finite partition $ \{ A_1, \dots, A_n\} $   of $Y$,  and
a finite set $ \{ \varphi_1, \dots, \varphi_m \} $  in  $  Y$  such that for all $ f \in M$, 
there is  $ j \in \{1, \dots, m\}$ with  
$\mbox{diam}( (f - \varphi_j) (A_i) ) \le 1 - \delta$  for $ i= 1, \dots, n$.
We set, for $ j \in \{1, \dots, m\}$,
\[
M_j= \{ f \in M : \  \mbox{diam}( (f - \varphi_j) (A_i) ) \le 1- \delta  \ \mbox{for} \  i= 1, \dots, n \},
\]
without loss of generality, we may assume that each $M_j $ is an infinite set.
Moreover, since $\omega (M)=2$,     there is  $f_s \in M$  and  $i \in \{ 1, \dots, n \}$ such that 
$
\mbox{diam} (f_s (A_i)) \ge  2 -  \delta. 
$
Fix $y,z \in A_i$ such that $\| f_s(y)-f_s(z) \| \ge  2 -  \delta$, and let  $j \in  \{ 1, \dots, m \}$ such that $f_s \in M_j$.
Then 
\[
1 - \delta \ge  \|  (f_s  - \varphi_j)(y)-    (f_s - \varphi_j)(z)  ) \| \ge \big| \|  f_s(y )- f_s(z) \| -  \| \varphi_j(y)- \varphi_j(z) \|  \big|, 
\]
so it follows  $ \| \varphi_j(y)-  \varphi_j(z)  \|  \ge 1 $.
On the other hand, taking $f_l    \in M_j$ with $l \ne s$ we have $f_l(y)=f_l(z)=0$,  thus
\[
\|f_l (y)- \varphi_j(y)-  (  f_l    (z)- \varphi_j(z)) \| = \|    \varphi_j(y)- \varphi_j(z)) \|  \le 1 - \delta,
\]
which is a contradiction, consequently 
  $\omega_\B (M)=1$ and  thus $\alpha (M)= \omega_\B(M)$.
  }
\end{example}

\begin{example}
{\rm
 Let $\Omega= [0, + \infty)$ and $Y= (\C([0,1]), \| \cdot \|_\infty)$,   hence $Y$ a Lindenstrauss space.  Now we  define
 $f_k : \Omega \to Y$  for $k=1, 2, \dots $, by putting
\[
f_k(x)= \begin{cases}
\psi_n  \quad \mbox{for}  \ \ x= k - \frac 1 n \ \ \mbox{for} \ \ n=1,2, \dots
\\  
\psi_0 \quad \mbox{otherwise} ,
\end{cases}
\]
with $\|\psi_n \|_\infty=1$ for $n=1,2, \dots$ and $\psi_0(t)=0$ for all $t \in [0,1]$, so that $f_k \in \B( \Omega, Y)$.
  Setting  $M= \{ f_k, \ \ k=1,2,\dots \}$ we have that $M$ is pointwise relatively compact.
Moreover,    $\|f_k-f_s\|_\infty = \sup_{x \in \Omega} \|f_k(x)-f_s(x) \|_\infty
=1$ (for $k\ne s$) and $\mbox{diam} (M)= \sup_{k,s \in \N} \|f_k-f_s\|_\infty~=~1$, hence similarly as  in the case of  Example~\ref{Lin 1},  we deduce $\alpha (M)=1$.

On the other hand, let us consider, for $n=1,2, \dots $ and  $t \in [0,1]$,
\[
\psi_n (t) = \begin{cases}
0   \quad & \mbox{for}  \ \ t \le 1 - \frac 1 n  
\\  
n\left(t- \left( 1- \frac 1 n \right)\right)  & \mbox{for}  \ \ 1- \frac 1 n  < t  \le 1,
\end{cases}
\]
and put  $\varphi (x)= \frac 1 2 \sum_{k=1}^\infty f_k(x)$ for $x \in \Omega$.  In order to  evaluate $\omega_\B (M)$ we consider $\{\Omega \}$   as a partition of $\Omega$  and
$ \{ \varphi \}$ as 
a finite set in $\B$. Then it is easy to check that
\[
\|f_k- \varphi \|_\infty \le \frac 1 2,
\]
whence { $\mbox{diam} ((f- \varphi)(Y)) \le \frac 1 2$, so that} $\omega_\B(M) \le \frac 1 2$. Since  $\alpha (M)=1$ and Theorem~\ref{Lin}  implies $ \frac 1 2 \alpha (M)  \le \omega_\B(M)$,   we have  $\alpha (M)= 2\omega_\B(M)$.
}
\end{example}

\noindent {\bf Acknowledgements.}
The first author was supported by FFR2018/Universit\`a degli Studi di Palermo.

\vspace{2ex}

Diana Caponetti\\
Dipartimento di Matematica e Informatica, 
Universit\`a di Palermo, Via Archirafi 34, 90123 Palermo, Italy. diana.caponetti@unipa.it

Alessandro Trombetta, Giulio Trombetta
\\ Dipartimento di Matematica e Informatica, 
Universit\`a  della Calabria,
Ponte Pietro Bucci 31B, 
87036 Arcavacata di Rende, Cosenza, Italy. aletromb@unical.it (A.T.);  \ trombetta@unical.it (G.T.)

\end{document}